\documentclass[a4paper,12pt,reqno]{amsart}
\usepackage{lmodern,amsmath,amssymb,mathtools,amsthm}
\usepackage{enumerate}
\usepackage{graphics,graphicx}
\usepackage{
cleveref}
\usepackage{bigints}
\usepackage{floatrow,morefloats}
\usepackage{epsfig}
\usepackage{epstopdf}
\usepackage{caption,subcaption}
\usepackage{geometry}   
\usepackage{wasysym}

\newtheorem{definition}{Definition}
\newtheorem{lemma}{Lemma}
\newtheorem{theorem}{Theorem}
\newtheorem{corollary}{Corollary}
\newtheorem{remark}{Remark}

\newtheorem*{problem}{Problem}
\numberwithin{equation}{section}
\numberwithin{definition}{section}
\numberwithin{lemma}{section}
\numberwithin{theorem}{section}
\numberwithin{corollary}{section}
\numberwithin{example}{section}

\usepackage{blindtext}

\title{Starlike And Convex Functions Associated with A Nephroid domain having Cusps On The Real Axis}

\author{Lateef Ahmad Wani$^\dagger$}
   \address{$^\dagger$Department of Mathematics\\ Indian Institute of Technology, Roorkee-247667, Uttarakhand, India}
   \email{lateef17304@gmail.com}

\author{A. Swaminathan$^\ddagger$}
   \address{$^\ddagger$Department of Mathematics\\ Indian Institute of Technology, Roorkee-247667, Uttarakhand, India}
   \email{swamifma@iitr.ac.in, mathswami@gmail.com}

\allowdisplaybreaks
\bigskip

\begin{document}
\begin{abstract}
In this paper, we show that the Carath\'{e}odory function $\varphi_{\scriptscriptstyle {Ne}}(z)=1+z-z^3/3$ maps the open unit disk $\mathbb{D}$ onto the interior of the nephroid, a $2$-cusped kidney-shaped curve,
\begin{align*}
\left((u-1)^2+v^2-\frac{4}{9}\right)^3-\frac{4 v^2}{3}=0,
\end{align*}
and introduce new Ma-Minda type function classes $\mathcal{S}^*_{Ne}$ and $\mathcal{C}_{Ne}$ associated with it. Apart from studying the characteristic properties of the region bounded by this nephroid, the structural formulas, extremal functions, growth and distortion results, inclusion results, coefficient bounds and Fekete-Szeg\"{o} problems are discussed for the classes $\mathcal{S}^*_{Ne}$ and $\mathcal{C}_{Ne}$. Moreover, for $\beta\in\mathbb{R}$ and some analytic function $p(z)$ satisfying $p(0)=1$, we prove certain subordination implications of the first order differential subordination $1+\beta\frac{zp'(z)}{p^j(z)}\prec\varphi_{\scriptscriptstyle {Ne}}(z),\,j=0,1,2,$ and obtain sufficient conditions for some geometrically defined function classes available in the literature.
\end{abstract}

\maketitle

\markboth{Lateef Ahmad Wani and A. Swaminathan}{Starlike and Convex Functions Associated with a Nephroid}

\section{Introduction}
Let $\mathbb{C}$ be the complex plane and $\mathbb{D}:=\left\{z\in\mathbb{C}:|z|<1\right\}$ be the open unit disk. Let $\mathcal{H}:=\mathcal{H}(\mathbb{D})$ be the collection of all analytic functions defined on $\mathbb{D}$, and let $\mathcal{A}$ consist of functions $f\in\mathcal{H}$ satisfying $f(0)=f'(0)-1=0$. Further, let $\mathcal{S}$ be the family of functions $f\in\mathcal{A}$ thet are univalent in $\mathbb{D}$. For $0\leq\alpha<1$, let $\mathcal{S}^*(\alpha)$ and $\mathcal{C}(\alpha)$ be the subclasses of $\mathcal{A}$ which consist of functions that are, respectively, starlike and convex of order $\alpha$. Analytically, these classes are represented as
\begin{align*}
\mathcal{S}^*(\alpha):=\left\{f\in\mathcal{A}:\mathrm{Re}\left(\frac{zf'(z)}{f(z)}\right)>\alpha\right\} \text{ and }
\mathcal{C}(\alpha):=\left\{f\in\mathcal{A}:\mathrm{Re}\left(1+\frac{zf''(z)}{f'(z)}\right)>\alpha\right\}.
\end{align*}
The classes $\mathcal{S}^*=\mathcal{S}^*(0)$ and $\mathcal{C}=\mathcal{C}(0)$ are the well-known classes of starlike and convex functions. These two classes are related by the familiar Alexander's theorem as: $f\in\mathcal{C}(\alpha)$ if and only if $zf'\in\mathcal{S}^*(\alpha)$. For $0<\beta\leq1$, the classes $\mathcal{SS}^*(\beta)\subset\mathcal{A}$ and $\mathcal{SC}(\beta)\subset\mathcal{A}$, consisting of strongly starlike and strongly convex functions of order $\beta$, are defined as
\begin{align*}
\mathcal{SS}^*(\beta):=\left\{f:\left|\arg \frac{zf'(z)}{f(z)}\right|<\frac{\beta\pi}{2}\right\}
\text{ and }
\mathcal{SC}(\beta):=
              \left\{f:\left|\arg\left(1+\frac{zf''(z)}{f'(z)}\right)\right|<\frac{\beta\pi}{2}\right\}.
\end{align*}
Observe that $\mathcal{SS}^*(1)=\mathcal{S}^*$, and for $0<\beta<1,\, \mathcal{SS}^*(\beta)$ consists only of bounded starlike functions, and hence in this case the inclusion $\mathcal{SS}^*(\beta)\subset\mathcal{S}^*$ is proper.\\
For$f,g\in\mathcal{H}$, the function $f$ is said to be subordinate to $g$, written as $f\prec g$, if there exists a function $w\in\mathcal{H}$ satisfying $w(0)=0$ and $|w(z)|<1$ such that $f(z)=g(w(z))$. Indeed, $f\prec g$ implies that $f(0)=g(0)$ and $f(\mathbb{D})\subset g(\mathbb{D})$. Moreover, if the function $g(z)$ is univalent, then $f\prec g$ if, and only if, $f(0)=g(0)$ and $f(\mathbb{D})\subset g(\mathbb{D})$. Here, the function $w(z)$ is the well known Schwarz function. An analytic function $f:\mathbb{D}\to\mathbb{C}$ satisfying $f(0)=1$ and $\mathrm{Re}\left(f(z)\right)>0$ for every $z\in\mathbb{D}$ is called a Carath\'{e}odory function.

Ma and Minda \cite{Ma-Minda-1992-A-unified-treatment} used the concept of subordination to develop an interesting method of constructing subclasses of starlike and convex functions. For this purpose, we consider the analytic function $\varphi:\mathbb{D}\to\mathbb{C}$ satisfying the following:
\begin{enumerate}[(i)]
	\item
	$\varphi(z)$ is univalent with $\mathrm{Re}(\varphi)>0$,
	\item
	$\varphi(\mathbb{D})$ is starlike with respect to $\varphi(0)=1$,
	\item
	$\varphi(\mathbb{D})$ is symmetric about the real axis, and
	\item
	$\varphi'(0)>0$.
\end{enumerate}
Throughout this manuscript, wherever the analytic function $\varphi$ is given, it implies that the function $\varphi$ retains the above properties. Using this $\varphi$, the authors in \cite{Ma-Minda-1992-A-unified-treatment} defined the function classes $\mathcal{S}^*(\varphi)$ and $\mathcal{S}^*(\varphi)$ as
\begin{align}\label{Definition-Ma-Minda-classes-varphi}
\mathcal{S}^*(\varphi):=\left\{f\in\mathcal{A}:\frac{zf'(z)}{f(z)}\prec\varphi(z)\right\} \text{ and }
\mathcal{C}(\varphi):=\left\{f\in\mathcal{A}:1+\frac{zf''(z)}{f'(z)}\prec\varphi(z)\right\}.
\end{align}
It is clear that for every $\varphi$ satisfying the conditions (i)--(iv), the classes $\mathcal{S}^*(\varphi)$ and $\mathcal{C}(\varphi)$ are the subclasses of $\mathcal{S}^*$ and $\mathcal{C}$, respectively. Specialization of the function $\varphi(z)$ in (\ref{Definition-Ma-Minda-classes-varphi}) leads to a number of well-known function classes. For instance, taking $\varphi(z)=(1+z)/(1-z)$ yields $\mathcal{S}^*$ and $\mathcal{C}$, taking $\varphi(z)=(1+(1-2\alpha)z)/(1-z)\,(0\leq\alpha<1)$ yields $\mathcal{S}^*(\alpha)$ and $\mathcal{C}(\alpha)$. If we set $\varphi(z)=(1+Az)/(1+Bz)$, where $A,B\in[-1,1]$ and $B<A$, we obtain the Janowski classes $\mathcal{S}^*[A,B]$ and $\mathcal{C}[A,B]$
(see \cite{Janowski-1973-class-some-extremal}). The classes $\mathcal{SS}^*(\beta)$ and $\mathcal{SC}(\beta)$ are obtained for $\varphi(z)=((1+z)/(1-z))^{\beta}\;(0<\beta\leq1)$. The parabolic starlike class $\mathcal{S}_P$ introduced by R{\o}nning \cite{Ronning-1993-UCV-PAMS} and the uniformly convex class UCV introduced by Goodman \cite{Goodman-1991-UCV-Ann.Polo} are obtained for the function
\begin{align*}
\varphi(z)=1+\frac{2}{\pi^2}\left(\log\frac{1+\sqrt{z}}{1-\sqrt{z}}\right)^2, \quad z\in\mathbb{D}.
\end{align*}
For $0\leq{k}<\infty$, the classes $k-\mathcal{ST}$ (k-uniformly starlike functions) and $k-UCV$ (k-uniformly convex functions) introduced by Kanas and Wisniowska \cite{Kanas-Wisni-1999-ConicRegions-k-UCV-JCAM} are obtained from (\ref{Definition-Ma-Minda-classes-varphi}) on taking the function $\varphi(z)$ as
\begin{align*}
\varphi(z)=p_k(z)=\frac{1}{1-k^2}\cosh\left(A\,\log\frac{1+\sqrt{z}}{1-\sqrt{z}}\right)-\frac{k^2}{1-k^2},
                       \quad z\in\mathbb{D},\, 0\leq{k}<\infty
\end{align*}
where $A=(2\cos^{-1}{k})/\pi$. Sok\'{o}{\l} and Stankiewicz \cite{Sokol-J.Stankwz-1996-Lem-of-Ber} introduced and discussed the starlike class $\mathcal{S}^*_L$ associated with the right-half of the lemniscate of Bernoulli $\left(u^2+v^2\right)^2-2\left(u^2-v^2\right)=0$. In Ma-Minda's form, $\mathcal{S}^*_L:=\mathcal{S}^*(\sqrt{1+z})$. Moreover, a function $f\in\mathcal{S}^*_L$ is called a Sok\'{o}{\l} and Stankiewicz starlike function. Sok\'{o}{\l} \cite{Sokol-1998-Cassinian-Curve} introduced another important class $\mathcal{S}^*_{q_c}:=\mathcal{S}^*(q_c)$, where $q_c(z)=\sqrt{1+cz}$ with $c\in(0,1]$. For $c\in(0,1)$, the function $q_c(z)$ maps $\mathbb{D}$ onto the interior of right loop of the Cassinian ovals $\left(u^2+v^2\right)^2-2\left(u^2-v^2\right)=c^2-1$. The following Ma-Minda type classes have been introduced in the recent past:
\begin{enumerate}[(a)]
\item
The class $\mathcal{S}^*_{RL}:=\mathcal{S}^*(\varphi_{\scriptscriptstyle{RL}})$ with
\begin{align*}
\varphi_{\scriptscriptstyle{RL}}(z)=\sqrt{2}-(\sqrt{2}-1)\sqrt{\frac{1-z}{1+2(\sqrt{2}-1)z}},
\end{align*}
was considered by Mendiratta et al. \cite{Mendiratta-2014-Shifted-Lemn-Bernoulli}. The function $\varphi_{\scriptscriptstyle{RL}}(z)$ maps $\mathbb{D}$ onto the region enclosed by the left-half of the shifted lemniscate of Bernoulli $\left((u-\sqrt{2})^2+v^2\right)^2-2\left((u-\sqrt{2})^2+v^2\right)=0$.
\item
The function class $\mathcal{S}^*_{\leftmoon}:=\mathcal{S}^*(z+\sqrt{1+z^2})$ was introduced by Raina and Sok\'{o}{\l}
\cite{Raina-Sokol-2015-Crescent-Shaped-I} and then further discussed in \cite{Gandhi-Ravi-2017-Lune,Raina-Sharma-Sokol-2018-Crescent-Shaped,
Sharma-Raina-Sokol-2019-Ma-Minda-Crescent-Shaped}. The function $\varphi_{\scriptscriptstyle{\leftmoon}}(z)=z+\sqrt{1+z^2}$ maps $\mathbb{D}$ onto the crescent shaped region $\left\{w\in\mathbb{C}:|w^2-1|<2|w|,\;\mathrm{Re}\,w>0\right\}$.
\item
The class $\mathcal{S}^*_{e}:=\mathcal{S}^*(e^{z})$ was introduced and discussed by Mendiratta et al. \cite{Mendiratta-Ravi-2015-Expo-BMMS}.
\item
Sharma et al. \cite{Sharma-Ravi-2016-Cardioid} introduced and investigated the class  $\mathcal{S}^*_C:=\mathcal{S}^*(1+4z/3+2z^2/3)$ associated with the cardioid $(9u^2+9v^2-18u+5)^2-16(9u^2+9v^2-6u+1)=0$, a {\it heart-shaped} curve.
\item
The function class $\mathcal{S}^*_{R}:=\mathcal{S}^*(\varphi_{\scriptscriptstyle{0}})$, where $\varphi_{\scriptscriptstyle{0}}(z)$ is the rational function
\begin{align*}
\varphi_{\scriptscriptstyle{0}}(z)=
     1+\frac{z}{k}\left(\frac{k+z}{k-z}\right)=1+\frac{1}{k}z+\frac{2}{k^2}z^2+\frac{2}{k^3}z^3+\cdots,
                        \quad k=1+\sqrt{2},
\end{align*}
was discussed by Kumar and Ravichandran \cite{Kumar-Ravi-2016-Starlike-Associated-Rational-Function}.
\item
The class $\mathcal{S}^*_{lim}:=\mathcal{S}^*(1+\sqrt{2}z+z^2/2)$ associated with the limacon $(4u^2+4v^2-8u-5)^2+8(4u^2+4v^2-12u-3)=0$ was considered by Yunus et al. \cite{Yunus-2018-Limacon}.
\item
Recently, Kargar et al. \cite{Kargar-2019-Booth-Lem-A.M.Physics} discussed the following starlike class associated with the Booth lemniscate:
\begin{align*}
\mathcal{BS}(\alpha):=\mathcal{S}^*\left(1+\frac{z}{1-\alpha{z}^2}\right), \quad 0\leq\alpha<1.
\end{align*}
\item
Cho et al. \cite{Cho-2019-Sine-BIMS} introduced the Ma-Minda type function class $\mathcal{S}^*_S:=\mathcal{S}^*(1+\sin{z})$ associated with the sine function.
\item
For $0\leq\alpha<1$, Khatter et al. \cite{Khatter-Ravi-2019-Lem-Exp-Alpha-RACSAM} introduced and discussed in detail the classes
\begin{align*}
\mathcal{S}^*_{\alpha,e}:=\mathcal{S}^*\left(\alpha+(1-\alpha)e^z\right) \text{ and }
\mathcal{S}^*_L(\alpha):=\mathcal{S}^*\left(\alpha+(1-\alpha)\sqrt{1+z}\right),
\end{align*}
associated with the exponential function and the lemniscate of Bernoulli. Clearly, for $\alpha=0$, these classes reduce to the function classes $\mathcal{S}^*_e$ and $\mathcal{S}^*_L$, respectively.
\item
Very recently, Goel and Kumar \cite{Goel-Siva-2019-Sigmoid-BMMS} introduced the starlike class $\mathcal{S}^*_{SG}:=\mathcal{S}^*(\varphi_{\scriptscriptstyle{SG}})$ associated with the sigmoid function $\varphi_{\scriptscriptstyle{SG}}(z)=2/(1+e^{-z})$.
\end{enumerate}

Motivated by the aforementioned works, in this paper, we introduce the classes $\mathcal{S}^*_{Ne}:=\mathcal{S}^*(\varphi_{\scriptscriptstyle{Ne}})$ and $\mathcal{C}_{Ne}:=\mathcal{C}(\varphi_{\scriptscriptstyle{Ne}})$, where the analytic function
$\varphi_{\scriptscriptstyle{Ne}}:\mathbb{D}\to\mathbb{C}$ is defined as $\varphi_{\scriptscriptstyle{Ne}}(z):=1+z-z^3/3$. First we show that $\mathcal{S}^*_{Ne}$ and $\mathcal{C}_{Ne}$ are well defined, i.e., $\varphi_{\scriptscriptstyle{Ne}}(z)$ satisfies the conditions (i)--(iv). For $z_1,z_2\in\mathbb{D}$, $\varphi_{\scriptscriptstyle {Ne}}(z_1)=\varphi_{\scriptscriptstyle {Ne}}(z_2)$ gives $(z_1-z_2)(z_1^2+z_1z_2+z_2^2-3)=0$. This holds true only if $z_1=z_2$, and hence, proves the univalency of $\varphi_{\scriptscriptstyle{Ne}}(z)$ in $\mathbb{D}$. Also, for each $z\in\mathbb{D}$, the function $L(z)=z-z^3/3$ satisfies
\begin{align*}
\mathrm{Re}\left(\frac{zL'(z)}{L(z)}\right)
   =\mathrm{Re}\left(\frac{1-z^2}{1-z^2/3}\right)
    =1-\frac{2}{3}\mathrm{Re}\left(\frac{z^2}{1-z^2/3}\right)
    \geq 1-\frac{2}{3}\frac{|z|^2}{1-|z|^2/3}
    >1-1=0.
\end{align*}
This shows that the function $L(z)$ is starlike in $\mathbb{D}$ and consequently, the function $\varphi_{\scriptscriptstyle{Ne}}(z)=1+L(z)$ is starlike with respect to $\varphi_{\scriptscriptstyle{Ne}}(0)=1$. We note that the starlikeness of $L(z)$ in $\mathbb{D}$ is also clear in light of the Brannan's criteria for starlikeness of a univalent polynomial \cite[Theorem 2.3]{Brannan-1970-On-Univalent-Polynomials}. The condition $\varphi_{\scriptscriptstyle {Ne}}'(0)>0$ is easy to verify. For the remaining, we prove the following result.
\begin{theorem}
The function $\varphi_{\scriptscriptstyle {Ne}}(z)=1+z-z^3/3$ maps $\mathbb{D}$ onto the region bounded by the nephroid
\begin{align}\label{Equation-of-Nephroid}
\left((u-1)^2+v^2-\frac{4}{9}\right)^3-\frac{4 v^2}{3}=0,
\end{align}
which is symmetric about the real axis and lies completely inside the right-half plane $u>0$ {\rm(see \Cref{Figure-Nephroid})}.
\end{theorem}
\begin{proof}
It is enough to prove that for $t\in(-\pi,\pi]$, the curve $\varphi_{\scriptscriptstyle {Ne}}(e^{it})$ is the nephroid (\ref{Equation-of-Nephroid}). Indeed,
\begin{align*}
u+iv=\varphi_{\scriptscriptstyle {Ne}}(e^{it})=1+\cos{t}-\frac{1}{3}\cos{3t}+i\left(\sin{t}-\frac{1}{3}\sin{3t}\right)
\end{align*}
gives
\begin{align*}
u-1=\cos{t}-\frac{\cos{3t}}{3}  \text{ and }  v=\sin{t}-\frac{\sin{3t}}{3}.
\end{align*}
Squaring and adding these two equations, we have
\begin{align*}
(u-1)^2+v^2&=\frac{10}{9}-\frac{2}{3}\cos{2t}\\
             &=\frac{4}{9}+\frac{4}{3}\sin^2{t}\\
               &=\frac{4}{9}+\left(\frac{4}{3}\left(\frac{4}{3}\sin^3{t}\right)^2\right)^\frac{1}{3}\\
                 &=\frac{4}{9}+\left(\frac{4}{3}v^2\right)^\frac{1}{3},
\end{align*}
which implies (\ref{Equation-of-Nephroid}).
\end{proof}
\begin{figure}[H]
 \includegraphics{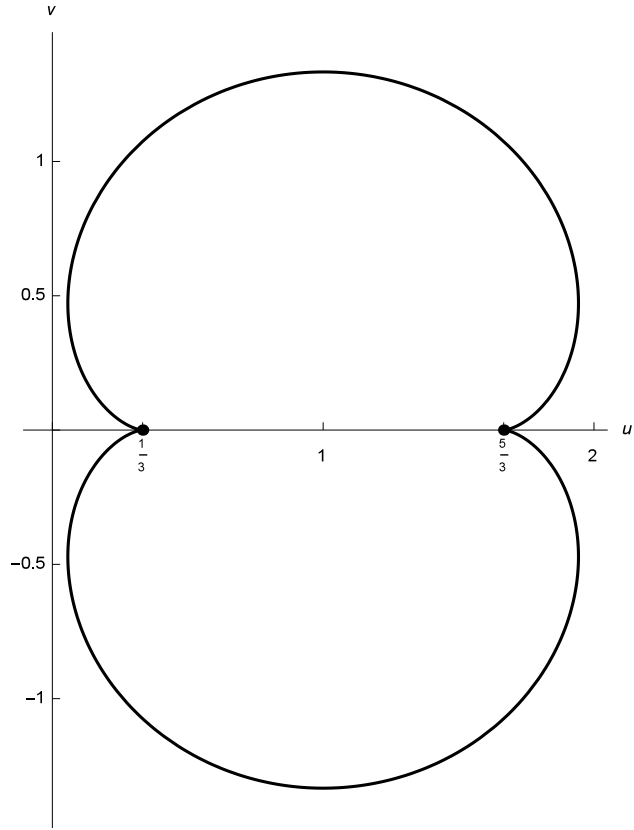}
 \caption{Nephroid: The boundary curve of $\varphi_{\scriptscriptstyle {Ne}}(\mathbb{D})$, where $\varphi_{\scriptscriptstyle {Ne}}(z)=1+z-z^3/3$.}
 \label{Figure-Nephroid}
\end{figure}
Thus, $\mathcal{S}^*_{Ne}$ and $\mathcal{C}_{Ne}$ are Ma-Minda type classes of analytic functions associated with the function $\varphi_{\scriptscriptstyle{Ne}}(z)=1+z-z^3/3$, and $f\in\mathcal{S}^*_{Ne}\;(\text{or } \mathcal{C}_{Ne})$ if, and only if, $zf'(z)/f(z)\;(\text{or } 1+zf''(z)/f'(z))$ lies inside the region bounded by the nephroid (\ref{Equation-of-Nephroid}).

Before proceeding further, we would like to provide some geometrical information and outline of the literature related to nephroid domain. Geometrically, a nephroid is the locus of a point on the circumference of a circle of radius $\rho$ traversing positively the outside of a fixed circle of radius $2\rho$. It is an algebraic curve of degree six and is an epicycloid having two cusps. The plane curve nephroid was studied by Huygens and Tschirnhausen around 1679 in connection with the theory of caustics, a method of deriving a new curve based on a given curve and a point. In 1692, J. Bernoulli showed that the nephroid is the catacaustic (envelope of rays emanating from a specified point) of a cardioid for a luminous cusp. However, the name nephroid, which means {\it kidney shaped}, was first used by the English mathematician Richard A. Proctor in 1878 in his book ``The Geometry of Cycloids". Further details in this direction can be found in \cite{Lockwood-Book-of-Curves-2007, Yates-1947-Handbook-Curves}.

We note that all the regions considered earlier were having either no cusp or a single one (cardioid), but in our case, the region bounded by the curve $\varphi_{\scriptscriptstyle {Ne}}(e^{it})\;(-\pi<t\leq\pi)$ consists of two cusps. Since, cusps (inward) of a function are due to its critical points, it follows that the cusps of $\varphi _{\scriptscriptstyle {Ne}}(z)$ are at the points $z=\pm1$. We further remark that, nephroid curves with both cusps on the imaginary axis, at the points $z=\pm i$,  are also used in the literature for various other problems. For our objective, we only consider the nephroid domains having their cusps on the real axis, at the points $z=\pm1$.
\section{Preliminary Lemmas and Structural Formula}  
\begin{lemma}\label{Lemma-Bounds-Re(varphiNe)-NeP}
For $0<r<1$, the function $\varphi_{\scriptscriptstyle {Ne}}(z)$ satisfies
\begin{align*}
\min_{|z|=r}\mathrm{Re}\big(\varphi_{\scriptscriptstyle {Ne}}(z)\big)=
\begin{cases}
1-r+\frac{1}{3}r^3,                            & \text{ if }~ r\leq \frac{1}{\sqrt{3}}\\
1-\frac{1}{3}\left(1+r^2\right)^\frac{3}{2},     & \text{ if }~ r\geq \frac{1}{\sqrt{3}}
\end{cases}
\end{align*}
and
\begin{align*}
\max_{|z|=r}\mathrm{Re}\big(\varphi_{\scriptscriptstyle {Ne}}(z)\big)=
\begin{cases}
1+r-\frac{1}{3}r^3,                           & \text{ if }~  r\leq \frac{1}{\sqrt{3}}\\
1+\frac{1}{3}\left(1+r^2\right)^\frac{3}{2},    & \text{ if }~  r\geq \frac{1}{\sqrt{3}}.
\end{cases}
\end{align*}
\end{lemma}
\begin{proof}
Let $z=re^{it}\;(-\pi<t\leq\pi)$. Then
\begin{align*}
\mathrm{Re}\left(\varphi_{\scriptscriptstyle {Ne}}(re^{it})\right)&=1+r\cos{t}-\frac{r^3}{3}\cos{3t}\\
          &=1+r(1+r^2)\cos{t}-\frac{4}{3}r^3\cos^3{t}\\
             &=1+r(1+r^2)x-\frac{4}{3}r^3x^3=:G(x),
\end{align*}
where $x=\cos{t}$. It is easy to verify that $G'(x)=0$ if, and only if, $x=\pm\sqrt{1+r^2}/2r$, and the number $\sqrt{1+r^2}/{2r}$ is less than or equal to $1$ if, and only if, $r\geq1/\sqrt{3}$. Further calculations show that $G(x)$ attains its minimum at $x=-\sqrt{1+r^2}/2r$ and maximum at $x=\sqrt{1+r^2}/2r$. Therefore, for $1/\sqrt{3}\leq{r}<1$,
\begin{align*}
\min_{|z|=r}\mathrm{Re}\big(\varphi_{\scriptscriptstyle {Ne}}(z)\big)=G\left(-\frac{\sqrt{1+r^2}}{2r}\right)=1-\frac{1}{3} \left(1+r^2\right)^\frac{3}{2}
\end{align*}
and
\begin{align*}
\max_{|z|=r}\mathrm{Re}\big(\varphi_{\scriptscriptstyle {Ne}}(z)\big)=G\left(\frac{\sqrt{1+r^2}}{2r}\right)=1+\frac{1}{3} \left(1+r^2\right)^\frac{3}{2}.
\end{align*}
Moreover, when $0<r\leq1/\sqrt{3}$, a graphical observation shows that the function $G(x)$ is increasing for every $x\in[-1,1]$. This implies that whenever $0<r\leq1/\sqrt{3}$,
\begin{align*}
\min_{|z|=r}\mathrm{Re}\big(\varphi_{\scriptscriptstyle {Ne}}(z)\big)=G(-1)=1-r+\frac{1}{3}r^3
\end{align*}
and
\begin{align*}
\max_{|z|=r}\mathrm{Re}\big(\varphi_{\scriptscriptstyle {Ne}}(z)\big)=G(1)=1+r-\frac{1}{3}r^3.
\end{align*}
Hence, the proof is completed.
\end{proof}
Let us denote by $\Omega_{Ne}$ the region bounded by the nephroid (\ref{Equation-of-Nephroid}), i.e.,
\begin{align*}
\Omega_{Ne}:=\left\{u+iv:\left((u-1)^2+v^2-\frac{4}{9}\right)^3-\frac{4 v^2}{3}<0\right\}.
\end{align*}

\begin{lemma}\label{Lemma-Radius-ra-NeP}
Let $1/3<a<5/3$. Let $r_a$ and $R_a$ be given by
\begin{align*}
r_a=
\begin{cases}
a-\frac{1}{3}, & \frac{1}{3}<a\leq1\\
\frac{5}{3}-a,            & 1\leq{a}<\frac{5}{3}.
\end{cases}
\end{align*}
and
\begin{align*}
R_a=\frac{1}{9} \sqrt{\frac{2\left(1+18(a-1)^2\right)^{3/2}+9a(a-2) (28+9a(a-2))+169}{(a-1)^2}}.
\end{align*}
Then
\begin{align*}
\left\{w\in\mathbb{C}:|w-a|<r_a\right\}\subseteq\Omega_{Ne}\subseteq\left\{w\in\mathbb{C}:|w-a|<R_a\right\}.
\end{align*}
\end{lemma}
\begin{proof}
For $z=e^{it}$, the parametric equations of the nephroid $\varphi_{\scriptscriptstyle{Ne}}(z)=1+z-z^3/3$ are
\begin{align*}
u(t)=1+\cos{t}-\frac{1}{3}\cos{3t}, \quad v(t)=\sin{t}-\frac{1}{3}\sin{3t}, \quad -\pi<t\leq\pi.
\end{align*}
The square of the distance from the point $(a,0)$ to the points on the nephroid (\ref{Equation-of-Nephroid}) is given by
\begin{align*}
z(t)&=(a-u(t))^2+(v(t))^2\\
     &=\frac{16}{9}+(a-1)^2-4(a-1)\cos{t}-\frac{4}{3}\cos^2{t}+\frac{8}{3}(a-1)\cos^3{t}\\
      &=\frac{16}{9}+(a-1)^2-4(a-1)x-\frac{4}{3}x^2+\frac{8}{3}(a-1)x^3=:H(x),
\end{align*}
where $x=\cos{t},\,-\pi<t\leq\pi$. Since the curve (\ref{Equation-of-Nephroid}) is symmetric about the real axis, it is sufficient to consider $0\leq{t}\leq\pi$. A simple computation shows that $H'(x)=0$ if, and only if,
\begin{align*}
x=\frac{1\pm\sqrt{1+18(a-1)^2}}{6(a-1)}.
\end{align*}
For $1/3<a<5/3$, only the number $x=x_0=\left(1-\sqrt{1+18(a-1)^2}\right)/6(a-1)$ lies between $-1$ and $1$. Further,
\begin{align*}
H''(x_0)=-\frac{8}{3}\sqrt{1+18(a-1)^2}<0 ~\text{ for each } ~a.
\end{align*}
This shows that $x_0$ is the point of maxima for the function $H(x)$ and hence, $H(x)$ is increasing in the interval $[-1,x_0]$ and decreasing in $[x_0,1]$. Therefore, for $1/3<a<5/3$, we have
\begin{align*}
\min_{0\leq{t}\leq\pi}{z(t)}=\min\left\{H(-1),H(1)\right\}
\end{align*}
and
\begin{align*}
\max_{0\leq{t}\leq\pi}{z(t)}=H(x_0).
\end{align*}
Since
\begin{align*}
H(1)-H(-1)=-\frac{8(a-1)}{3},
\end{align*}
so that $H(-1)\leq{H(1)}$ whenever $a\leq1$ and $H(1)\leq{H(-1)}$ whenever $a\geq1$, we conclude that
\begin{align*}
r_{a}=\min_{0\leq{t}\leq\pi}\sqrt{z(t)}=
\begin{cases}
\sqrt{H(-1)}=a-\frac{1}{3}, &\text{ whenever } \frac{1}{3}<a\leq1\\
\sqrt{H(1)}=\frac{5}{3}-a,   &\text{ whenever } 1\leq{a}<\frac{5}{3}.
\end{cases}
\end{align*}
Also,
\begin{align*}
R_{a}=\max_{0\leq{t}\leq\pi}\sqrt{z(t)}=\sqrt{H(x_0)}.
\end{align*}
This completes the proof of the lemma.
\end{proof}
We pose the following problem.
\begin{problem}
To prove that if $-4/3<a<4/3$ and
\begin{align*}
\hat{R}_a=
\begin{cases}
\frac{4}{3}-a, & -\frac{4}{3}<a\leq0\\
a+\frac{4}{3},            & 0\leq{a}<\frac{4}{3},
\end{cases}
\end{align*}
then
\begin{align*}
\Omega_{Ne}\subseteq\left\{w\in\mathbb{C}:|w-(1+ia)|<\hat{R}_a\right\}.
\end{align*}
\end{problem}
\begin{remark}
The largest disk centered at $z=1$ that is contained in $\Omega_{Ne}$ is $\left\{w:\left|w-1\right|<2/3\right\}$ and the smallest disk centered at $z=1$ that contains $\Omega_{Ne}$ is $\left\{w:\left|w-1\right|<4/3\right\}$.
\end{remark}
We now describe the structural formulas and the extremal functions for the classes $\mathcal{S}^*_{Ne}$ and $\mathcal{C}_{Ne}$, and construct some examples.
\begin{theorem}\label{Theorem-Structural-Formulas-NeP}
A function $f$ belongs to the class $\mathcal{S}^*_{Ne}$ if, and only if, there exists an analytic function $q(z)$, satisfying $q(z)\prec\varphi_{\scriptscriptstyle {Ne}}(z)=1+z-z^3/3$ such that
\begin{align}\label{Eq1-Structural-Formulas-NeP}
f(z)=z\exp\left(\int_0^z\frac{q(\zeta)-1}{\zeta}d\zeta\right), \quad z\in\mathbb{D}.
\end{align}
\end{theorem}
\begin{proof}
Let $f\in\mathcal{S}^*_{Ne}$. Then it follows from the definition of $\mathcal{S}^*_{Ne}$ that the function
\begin{align}\label{Eq2-Structural-Formulas-NeP}
q(z)=\frac{zf'(z)}{f(z)}
\end{align}
is analytic in $\mathbb{D}$ and $q(z)\prec\varphi_{\scriptscriptstyle {Ne}}(z)$. On simplifying (\ref{Eq2-Structural-Formulas-NeP}), we get
\begin{align*}
\frac{d}{dz}\left(\log\frac{f(z)}{z}\right)=\frac{q(z)-1}{z},
\end{align*}
which, by simple calculations provide (\ref{Eq1-Structural-Formulas-NeP}). Converse follows by retracing the procedure, from (\ref{Eq1-Structural-Formulas-NeP}).
\end{proof}
Using the Alexander type relation between $\mathcal{S}^*_{Ne}$ and $\mathcal{C}_{Ne}$, the following result is immediate.
\begin{corollary}
A function $f$ belongs to the class $\mathcal{C}_{Ne}$ if, and only if, there exists an analytic functions $q(z)$, satisfying $q(z)\prec\varphi_{\scriptscriptstyle {Ne}}(z)$ such that
\begin{align}\label{Eq3-Structural-Formulas-NeP}
f(z)=\int_{0}^{z}\left(\exp\int_0^\zeta\frac{q(\eta)-1}{\eta}d\eta\right)d\zeta, \quad z\in\mathbb{D}.
\end{align}
\end{corollary}
Next, consider the analytic functions $q_i:\mathbb{D}\to\mathbb{C}$, $i=1,2,3,4,5$, defined as follows.
\begin{align*}
q_1(z):&=\sqrt{1+cz} \text{ with } 0<c\leq8/9,\; q_2(z):=1+\frac{z}{2}\sin{z},\; q_3(z):=1+\frac{2}{9}z{e^z},\\
q_4(z):&=\frac{2}{1+e^{-z}} \text{ and } q_5(z):=\frac{5+3 z^2}{5+z^2}, \text{ for } z\in\mathbb{D}.
\end{align*}
For each $i\;(i=1,2,3,4,5)$, it is easy to verify that $q_i(0)=1=\varphi_{\scriptscriptstyle {Ne}}(0)$ and $q_i(\mathbb{D})\subset\varphi_{\scriptscriptstyle {Ne}}(\mathbb{D})$. Since the function $\varphi_{\scriptscriptstyle {Ne}}(z)=1+z-z^3/3$ is univalent, we conclude that $q_i(z)\prec\varphi_{\scriptscriptstyle {Ne}}(z)$ for every $i\;(i=1,2,3,4,5)$. Consequently, it follows from \Cref{Theorem-Structural-Formulas-NeP} that the functions $f_i(z)\;(i=1,2,3,4,5)$ given below are members of the class $\mathcal{S}^*_{Ne}$.
\begin{align*}
f_1(z)&=\frac{4z\exp\left(2\sqrt{1+cz}-2\right)}{\left(\sqrt{1+cz}+1\right)^2},\quad 0<c\leq{8}/{9},\\
f_2(z)&=z\exp\left(\frac{1-\cos{z}}{2}\right),\\
f_3(z)&=z\exp\left(\frac{2e^z-2}{9}\right),\\
f_4(z)&=z+\frac{z^2}{2}+\frac{z^3}{8}+\frac{z^4}{144}-\frac{5z^5}{1152}-\frac{37z^6}{57600}
              +\frac{509z^7}{2073600}+\cdots,\\
f_5(z)&=\frac{z}{5}\left(z^2+5\right).
\end{align*}
In particular, if we choose $q(z)=\varphi_{\scriptscriptstyle {Ne}}(z)=1+z-z^3/3$, then the function
\begin{align*}
f_{\scriptscriptstyle {Ne}}(z):=z{e^{z-\frac{z^3}{9}}}=z+z^2+\frac{z^3}{2}+\frac{z^4}{18}-\frac{5 z^5}{72}-\frac{17 z^6}{360}-\frac{71 z^7}{6480}+\cdots
\end{align*}
is a member of the class $\mathcal{S}^*_{Ne}$. In the sequel, we will see that the function $f_{\scriptscriptstyle {Ne}}(z)=z{e^{z-\frac{z^3}{9}}}$ plays the role of an extremal function for several extremal problems in the class $\mathcal{S}^*_{Ne}$. Furthermore, the corresponding functions belonging to the class $\mathcal{C}_{Ne}$ can be easily constructed using (\ref{Eq3-Structural-Formulas-NeP}). In particular, the function
\begin{align*}
\widehat{f}_{\scriptscriptstyle{Ne}}(z):=\int_{0}^{z}\frac{f_{\scriptscriptstyle {Ne}}(\zeta)}{\zeta}d\zeta
                        =\int_{0}^{z}e^{\zeta-\frac{\zeta^3}{9}}d\zeta
                                 =z+\frac{z^2}{2}+\frac{z^3}{6}+\frac{z^4}{72}-\frac{z^5}{72}-\frac{17 z^6}{2160}+\cdots
\end{align*}
is a member of the class $\mathcal{C}_{Ne}$. The function $\widehat{f}_{\scriptscriptstyle {Ne}}(z)$ plays central role in $\mathcal{C}_{Ne}$ for the extremal problems.

Following Ma and Minda \cite[Section 3]{Ma-Minda-1992-A-unified-treatment}, the following results are easy to establish and hence we omit the proof.
\begin{theorem}
Let $f\in\mathcal{S}^*_{Ne}$ and let $|z|=r<1$. Then
\begin{enumerate}[{\rm(i)}]
\item Subordination results: $\frac{zf'(z)}{f(z)}\prec\frac{zf'_{\scriptscriptstyle {Ne}}(z)}{f_{\scriptscriptstyle {Ne}}(z)}$ and $\frac{f(z)}{z}\prec\frac{f_{\scriptscriptstyle{Ne}}(z)}{z}$.
\item Growth Theorem: $-f_{\scriptscriptstyle {Ne}}(-r)\leq|f(z)|\leq{f_{\scriptscriptstyle {Ne}}(r)}$. Equality holds for some non-zero $z$ if, and only if, $f$ is a rotation of $f_{\scriptscriptstyle{Ne}}(z)$.
\item Covering Theorem: Either $f$ is a rotation of $f_{\scriptscriptstyle {Ne}}$ or $f(\mathbb{D})$ contains the disk $\Delta^*=\left\{w:|w|<-f_{\scriptscriptstyle {Ne}}(-1)=e^\frac{-8}{9}\approx0.411112\right\}$.
\item Rotation Theorem: $\left|\arg\frac{f(z)}{z}\right|\leq
    \displaystyle{\max_{|z|=r}}\;\arg\left(\frac{f_{\scriptscriptstyle{Ne}}(z)}{z}\right)$. Equality holds for some non-zero $z$ if, and only if, $f$ is a rotation of $f_{\scriptscriptstyle{Ne}}(z)$.
\end{enumerate}
\end{theorem}
\begin{figure}[h!]
 \includegraphics{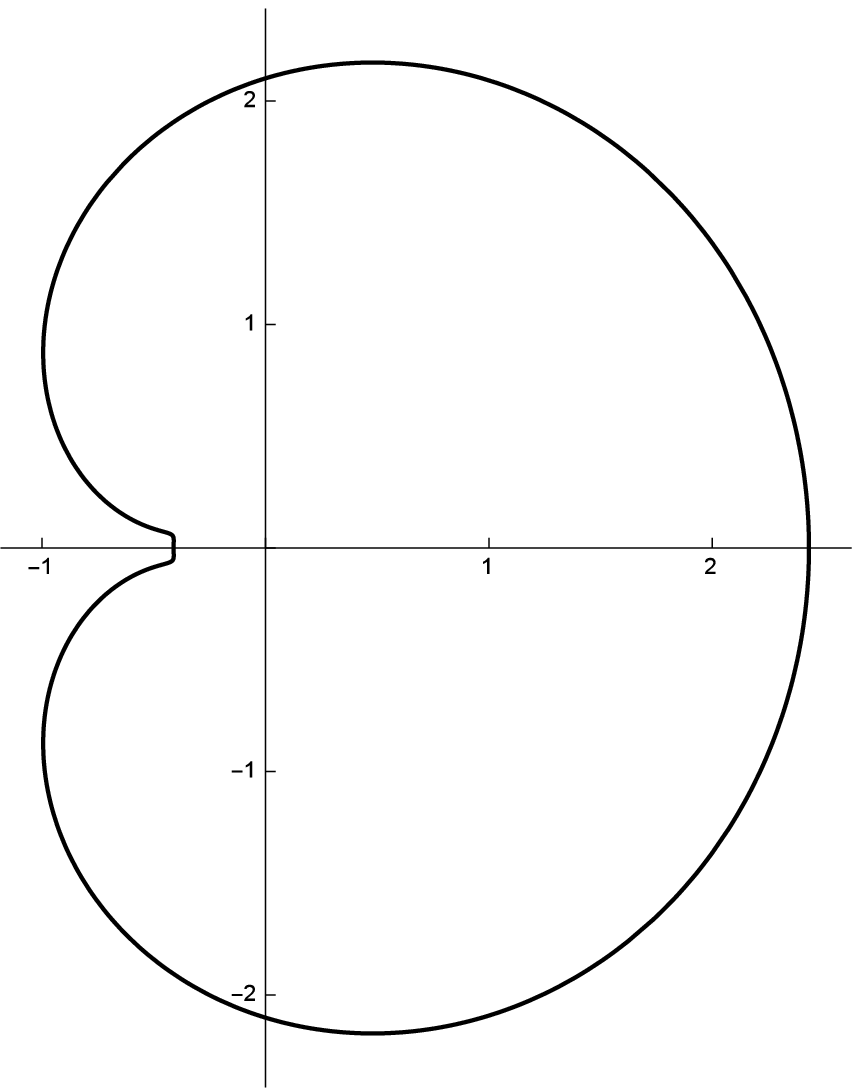}
 \caption{Boundary curve of $f_{\scriptscriptstyle{Ne}}(\mathbb{D})$, where $f_{\scriptscriptstyle{Ne}}(z)=z e^{z-\frac{z^3}{9}}$.}
 \label{Figure-Extremal-Function-fNe}
\end{figure}
\begin{theorem}
Let $f\in\mathcal{C}_{Ne}$ and let $|z|=r<1$. Then
\begin{enumerate}[{\rm(i)}]
\item Subordination results: $\frac{zf'(z)}{f(z)}\prec\frac{z\widehat{f}'_{\scriptscriptstyle {Ne}}(z)}{\widehat{f}_{\scriptscriptstyle {Ne}}(z)}$ and $\frac{f(z)}{z}\prec\frac{\widehat{f}_{\scriptscriptstyle{Ne}}(z)}{z}$.
\item Distortion Theorem: $\widehat{f}'_{\scriptscriptstyle {Ne}}(-r)\leq|f'(z)|\leq{\widehat{f}'_{\scriptscriptstyle {Ne}}(r)}$.
\item Growth Theorem: $-\widehat{f}_{\scriptscriptstyle {Ne}}(-r)\leq|f(z)|\leq{\widehat{f}_{\scriptscriptstyle {Ne}}(r)}$.
\item Covering Theorem: Either $f$ is a rotation of $\widehat{f}_{\scriptscriptstyle {Ne}}$ or $f(\mathbb{D})$ contains the disk $\Delta^c=\left\{w:|w|<-\widehat{f}_{\scriptscriptstyle {Ne}}(-1)\approx0.645159\right\}$.
\item Rotation Theorem: $\left|\arg\frac{f(z)}{z}\right|\leq
    \displaystyle{\max_{|z|=r}}\;\arg\left(\frac{\widehat{f}_{\scriptscriptstyle{Ne}}(z)}{z}\right)$.
\end{enumerate}
Equality in (ii), (iii) and (iv) holds for some non-zero $z$ if, and only if, $f$ is a rotation of $\widehat{f}_{\scriptscriptstyle{Ne}}(z)$.
\end{theorem}
\begin{remark}
In the class $\mathcal{S}^*_{Ne}$, the function $f_{\scriptscriptstyle{Ne}}(z)$ is not an extremal for the distortion theorem. It is evident from the function $f_{\scriptscriptstyle{Ne}}(z)$ itself, as the inequality
\begin{align*}
f'_{\scriptscriptstyle{Ne}}(-r)\leq|f'_{\scriptscriptstyle{Ne}}(z)|\leq{f'_{\scriptscriptstyle{Ne}}(r)}
\end{align*}
is not true for $|z|=r>1/\sqrt{3}$, whereas for $|z|\leq1/\sqrt{3}$, the result is true due to \Cref{Lemma-Bounds-Re(varphiNe)-NeP} and
\cite[Theorem 2]{Ma-Minda-1992-A-unified-treatment}.
\end{remark}

\section{Inclusion Results}   
In this section, we prove certain inclusion relationship results.
\begin{theorem}\label{Theorem-Inclusion-Relations-NeP}
The function class $\mathcal{S}^*_{Ne}$ satisfies the following inclusion properties:
\begin{enumerate}[{\rm(a)}]
\item\label{Inclusion-a} $\mathcal{S}^*_{Ne}\subset\mathcal{S}^*(\alpha)$ whenever $0\leq\alpha\leq1-\frac{2\sqrt{2}}{3}\approx{0.057191}$.
\item\label{Inclusion-c} $\mathcal{S}^*_{Ne}\subset\mathcal{SS}^*(\beta)$ whenever $\beta_0\leq\beta\leq{1}$, where $\beta_0\approx{0.929121}$.
\item\label{Inclusion-d} $\mathcal{S}^*_{q_c}\subset\mathcal{S}^*_{Ne}$ whenever $0<c\leq\frac{8}{9}$.
\item\label{Inclusion-e} $k-\mathcal{ST}\subset\mathcal{S}^*_{Ne}$ whenever $k\geq\frac{5}{2}$.
\item\label{Inclusion-f} $\mathcal{S}^*_{\alpha,e}\subset\mathcal{S}^*_{Ne}$ whenever $\alpha\geq1-\frac{2/3}{e-1}\approx{0.612016}$.
\item\label{Inclusion-g} $\mathcal{S}^*_L(\alpha)\subset\mathcal{S}^*_{Ne}$ whenever $\alpha\geq\frac{1}{3}$.
\end{enumerate}
The constant in each part is best possible.
\end{theorem}
\begin{proof}
\begin{enumerate}
\item[\eqref{Inclusion-a}:]
Let $f\in\mathcal{S}^*_{Ne}$. Then $zf'(z)/f(z)\prec\varphi_{\scriptscriptstyle{Ne}}(z)$. In view of \Cref{Lemma-Bounds-Re(varphiNe)-NeP}, it can be easily deduced that
\begin{align*}
1-\frac{2\sqrt{2}}{3}=\min_{|z|=1}\mathrm{Re}\left(\varphi_{\scriptscriptstyle{Ne}}(z)\right)
           <\mathrm{Re}\left(\frac{zf'(z)}{f(z)}\right).
\end{align*}
This implies that $f\in\mathcal{S}^*({1-2\sqrt{2}/3})$. Therefore,  $\mathcal{S}^*_{Ne}\subset\mathcal{S}^*({1-2\sqrt{2}/3})$ and proves (\ref{Inclusion-a})
\item[\eqref{Inclusion-c}:]
It is clear that $f\in\mathcal{S}^*_{Ne}$ implies
\begin{align*}
\left|\arg\left(\frac{zf'(z)}{f(z)}\right)\right|&
                 <\max_{|z|=1}\arg\left(\varphi_{\scriptscriptstyle{Ne}}(z)\right)\\
 &=\max_{t\in(-\pi,\pi]}\arg\left(\varphi_{\scriptscriptstyle{Ne}}(e^{it})\right)\\
      &=\max_{t\in(-\pi,\pi]}
      \tan^{-1}\left(\frac{\sin{t}-\frac{1}{3}\sin{3t}}{1+\cos{t}-\frac{1}{3}\cos{3t}}\right).
\end{align*}
Since $\tan^{-1}\theta\,(\theta\in\mathbb{R})$ is an increasing function, an easy computation leads us to the problem of finding $\tan^{-1}\left(\max Q(t)\right)$, where $Q(t)={4\sin^3{t}}/(3+6\cos{t}-4\cos^3{t})$.
Simple verification shows that
$Q(t)$ attains its maximum at $t=\cos^{-1}(-2/3)$ and hence, we conclude that  
\begin{align*}
\left|\arg\left(\frac{zf'(z)}{f(z)}\right)\right|<
                    \tan^{-1}\left(Q\left(\cos^{-1}(-2/3)\right)\right)\approx{1.45946}
                            \approx(0.929121)\frac{\pi}{2}.
\end{align*}
This implies that $f\in\mathcal{SS}^*(\beta_0)$, where $\beta_0\approx{0.929121}$, and establishes (\ref{Inclusion-c}).
\item[\eqref{Inclusion-d}:]
It has been proved \cite{Sokol-1998-Cassinian-Curve} that the function $q_c(z)=\sqrt{1+cz}\,(0<c\leq{1})$ maps $\mathbb{D}$ onto the interior of the right loop of the Cassini's curve $|w^2-1|=c$. A simple analysis shows that $q_c(\mathbb{D})$ is contained in $\Omega_{Ne}$ if, and only if, $\sqrt{1-c}\geq{1/3}$, i.e., if, and only if, $c\leq{8/9}$, which is nothing but (\ref{Inclusion-d}).
\item[\eqref{Inclusion-e}:]
If $f\in{k-\mathcal{ST}}$, then $zf'(z)/f(z)\prec{p_k(z)}$, where
\begin{align*}
p_k(z)=\frac{1}{1-k^2}\cosh\left(A\,\log\frac{1+\sqrt{z}}{1-\sqrt{z}}\right)-\frac{k^2}{1-k^2},
\end{align*}
with $A=\frac{2}{\pi}\cos^{-1}{k},\,0\leq{k}<\infty$. For $k>1$, the function $p_k(z)$ maps $\mathbb{D}$ onto the region $\Omega_{k}$ bounded by the ellipse
(cf. \cite[p. 331]{Kanas-Wisni-1999-ConicRegions-k-UCV-JCAM})
\begin{align*}
\frac{\left(u-\frac{k^2}{k^2-1}\right)^2}{\left(\frac{k}{k^2-1}\right)^2}+
              \frac{v^2}{\frac{1}{k^2-1}}=1.
\end{align*}
For the region $\Omega_{k}$ to lie inside $\Omega_{Ne}$, we should have
\begin{align*}
\frac{k^2}{k^2-1}+\frac{k}{k^2-1}\leq\frac{5}{3}.
\end{align*}
This gives $k\geq{5/2}$, as in (\ref{Inclusion-e})
\item[\eqref{Inclusion-f}:]
Since the function $\varphi_{\scriptscriptstyle{\alpha,e}}(z)=\alpha+(1-\alpha)e^z$ maps $\mathbb{D}$ onto the region
\begin{align*}
\left\{w:\left|\log\left(\frac{w-\alpha}{1-\alpha}\right)\right|<1\right\}.
\end{align*}
In view of \cite[Lemma 2.1]{Khatter-Ravi-2019-Lem-Exp-Alpha-RACSAM}, a simple observation shows that the above region will lie inside $\Omega_{Ne}$ provided
\begin{align*}
\alpha+(1-\alpha)e\leq\frac{5}{3}.
\end{align*}
This on further simplification gives $\alpha\geq{1-2/3(e-1)}$ and settles (\ref{Inclusion-f}).
\item[\eqref{Inclusion-g}:]
It is easy to verify that the function $\varphi_{\scriptscriptstyle{L,\alpha}}(z)=\alpha+(1-\alpha)\sqrt{1+z}$ maps $\mathbb{D}$ onto
\begin{align*}
\Delta^*_{L,\alpha}:=\left\{w:\left|\left(\frac{w-\alpha}{1-\alpha}\right)^2-1\right|<1\right\}.
\end{align*}
Again, in view of \cite[Lemma 2.1]{Khatter-Ravi-2019-Lem-Exp-Alpha-RACSAM}, it follows that $\Delta^*_{L,\alpha}$ lies in the interior of the region $\Omega_{Ne}$ if $\alpha\geq{1/3}$, which establishes (\ref{Inclusion-g}) and the proof is complete.    \qedhere
\end{enumerate}
\end{proof}

\begin{minipage}{0.55\textwidth}
\includegraphics{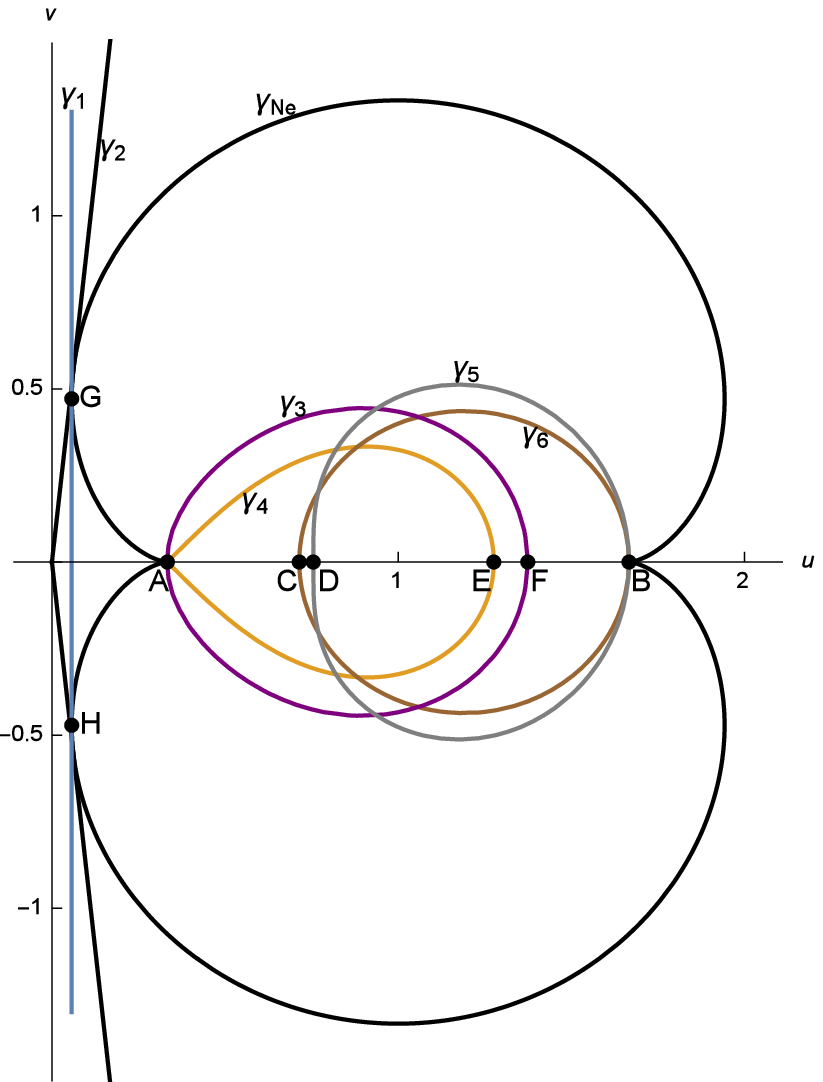}
\end{minipage}
\begin{minipage}{0.45\textwidth}
$\gamma_{Ne}:\left((u-1)^2+v^2-\frac{4}{9}\right)^3-\frac{4 v^2}{3}=0$.\\
$\gamma_1:\mathrm{Re}\,w=1-\frac{2\sqrt{2}}{3}$.\\
$\gamma_2:|\arg{w}|=\frac{\beta_0\pi}{2},\,\beta_0={0.929121}$.\\
$\gamma_3:\left|w^2-1\right|=\frac{8}{9},\mathrm{Re}\,w>0$.\\
$\gamma_4:\left|\left(\frac{w-\alpha}{1-\alpha}\right)^2-1\right|=1,\,\alpha=\frac{1}{3}$.\\
$\gamma_5:\left|\log\left(\frac{w-\alpha}{1-\alpha}\right)\right|=1,\mathrm{Re}\,w>0$\\
with $\alpha=1-\frac{2/3}{e-1}$.\\
 $\gamma_6:\mathrm{Re}\,w=\frac{5}{2}\left|w-1\right|$.\\
$A=\frac{1}{3}$\\
$B=\frac{5}{3}$\\
$C=\frac{5}{7}$\\
$D=1-\frac{2}{3e}$\\
$E=\frac{1+2\sqrt{2}}{3}$\\
$F=\frac{\sqrt{17}}{3}$\\
$G=\left(1-\frac{2\sqrt{2}}{3},\frac{\sqrt{2}}{3}\right)$\\
$H=\left(1-\frac{2\sqrt{2}}{3},-\frac{\sqrt{2}}{3}\right)$\\
\end{minipage}
\captionof{figure}{Boundary curves of best dominants and best subordinants of the nephroid.}
\label{Figure-Inclusion-Relations}
\noindent

The following theorem is an immediate consequence of \Cref{Theorem-Inclusion-Relations-NeP} obtained by using the Alexander type relation between $\mathcal{S}^*_{Ne}$ and $\mathcal{C}_{Ne}$.
\begin{theorem}
The class $\mathcal{C}_{Ne}$ satisfies the following properties:
\begin{enumerate}[{\rm(a)}]
\item $\mathcal{C}_{Ne}\subset\mathcal{C}(\alpha)$, whenever $0\leq\alpha\leq1-\frac{2\sqrt{2}}{3}$.
\item $\mathcal{C}_{Ne}\subset\mathcal{SC}(\beta)$, whenever $\beta_0\leq\beta\leq{1}$, where $\beta_0\approx{0.929121}$.
\item $\mathcal{C}(q_c)\subset\mathcal{C}_{Ne}$, whenever $0<c\leq\frac{8}{9}$, where $\mathcal{C}(q_c)$ is the convex class corresponding to the class $\mathcal{S}^*_{q_c}$.
\item $k-UCV\subset\mathcal{C}_{Ne}$, whenever $k\geq\frac{5}{2}$.
\item $\mathcal{C}_{\alpha,e}\subset\mathcal{C}_{Ne}$, whenever $\alpha\geq1-\frac{2/3}{e-1}$, where $\mathcal{C}_{\alpha,e}$ is the convex class corresponding to $\mathcal{S}^*_{\alpha,e}$.
\item $\mathcal{CL}(\alpha)\subset\mathcal{C}_{Ne}$, whenever $\alpha\geq\frac{1}{3}$, where $\mathcal{CL}(\alpha)$ is the convex class corresponding to the class $\mathcal{S}^*_L(\alpha)$.
\end{enumerate}
The constant in each part is best possible.
\end{theorem}
For $-1\leq{B}<A\leq{1}$, let $\mathcal{P}[A,B]$ be the totality of analytic functions $p:\mathbb{D}\to\mathbb{C}$ of the form $p(z)=1+\sum_{n=1}^{\infty}p_nz^n$ and satisfying the subordination $p(z)\prec{(1+Az)/(1+Bz)}$. Further, we set $\mathcal{P}[1-2\alpha,-1]=\mathcal{P}(\alpha)\,(0\leq\alpha<1)$ and $\mathcal{P}(0)=\mathcal{P}$. Also, recall \cite{Janowski-1973-class-some-extremal} that the Janowski starlike class $\mathcal{S}^*[A,B]$ consists of functions $f\in\mathcal{A}$ satisfying $zf'(z)/f(z)\in\mathcal{P}[A,B]$.
\begin{lemma}[{\cite[Lemma 2.1]{Ravi-Ronning-1997-Complex-Variables}}]\label{Lemma2.1-Ravi-Ronning-1997-P{A,B}}
If $p\in\mathcal{P}[A,B]$, then for $|z|=r<1$
\begin{align*}
\left|p(z)-\frac{1-ABr^{2}}{1-B^2r^2}\right|\leq\frac{(A-B)r}{1-B^2r^2}.
\end{align*}
For the particular case $p\in\mathcal{P}(\alpha)$, we have
\begin{align*}
\left|p(z)-\frac{1+(1-2{\alpha})r^{2}}{1-r^2}\right|\leq\frac{2(1-\alpha)r}{1-r^2}.
\end{align*}
\end{lemma}
\begin{theorem}\label{Theorem-Inclusion-Janowski-Class-NeP}
Let $-1<B<A\leq1$. Then $\mathcal{S}^*[A,B]\subset\mathcal{S}^*_{Ne}$, if the parameters $A$ and $B$ satisfy any one of the following conditions:
\begin{align}
(1-B^2)<3(1-AB)\leq3(1-B^2) &\text{ and } (1-B)\leq3(1-A). \tag{3a}\label{3a}\\
3(1-B^2)\leq3(1-AB)<5(1-B^2) &\text{ and } 3(1+A)\leq5(1+B). \tag{3b}\label{3b}
\end{align}
\end{theorem}
\begin{proof}
Let $f\in\mathcal{S}^*[A,B]$. Then $zf'(z)/f(z)\in\mathcal{P}[A,B]$, so that \Cref{Lemma2.1-Ravi-Ronning-1997-P{A,B}} gives
\begin{align}\label{Disk-Ravi-Ronning-P{A,B}-NeP}
\left|\frac{zf'(z)}{f(z)}-\frac{1-AB}{1-B^2}\right|\leq\frac{A-B}{1-B^2}.
\end{align}
The inequality (\ref{Disk-Ravi-Ronning-P{A,B}-NeP}) represents a disk with center at $(1-AB)/(1-B^2)$ and radius $(A-B)/(1-B^2)$. In order to prove our result, we need to show that the disk (\ref{Disk-Ravi-Ronning-P{A,B}-NeP}) lies inside $\Omega_{Ne}$. Since from \Cref{Lemma-Radius-ra-NeP} the disk
\begin{align}\label{Disk-Lemma-Radius-Incl-Pts-Nep}
\left|w-\dfrac{1-AB}{1-B^2}\right|<r_{a}=
\begin{cases}
\dfrac{1-AB}{1-B^2}-\frac{1}{3}, &\text{ if } ~ \dfrac{1}{3}<\dfrac{1-AB}{1-B^2}\leq1 \\
\\
\dfrac{5}{3}-\frac{1-AB}{1-B^2},            &\text{ if } ~ 1\leq{\dfrac{1-AB}{1-B^2}}<\dfrac{5}{3}
\end{cases}
\end{align}
is contained in the region $\Omega_{Ne}$, it is sufficient to show that the disk (\ref{Disk-Ravi-Ronning-P{A,B}-NeP}) lies inside the disk \eqref{Disk-Lemma-Radius-Incl-Pts-Nep}. This is true if
\begin{align*}
\dfrac{A-B}{1-B^2}\leq{r_a}=
\begin{cases}
\dfrac{1-AB}{1-B^2}-\dfrac{1}{3}, &\text{ if } ~ \dfrac{1}{3}<\dfrac{1-AB}{1-B^2}\leq1\\
\\
\dfrac{5}{3}-\dfrac{1-AB}{1-B^2},            &\text{ if } ~ 1\leq{\dfrac{1-AB}{1-B^2}}<\dfrac{5}{3},
\end{cases}
\end{align*}
or equivalently, if either
\begin{align*}
\frac{1}{3}<\frac{1-AB}{1-B^2}\leq1 \text{ and } \frac{A-B}{1-B^2}\leq\frac{1-AB}{1-B^2}-\frac{1}{3}
\end{align*}
or
\begin{align*}
1\leq{\frac{1-AB}{1-B^2}}<\frac{5}{3} \text{ and } \frac{A-B}{1-B^2}\leq\frac{5}{3}-\frac{1-AB}{1-B^2}
\end{align*}
hold true. Further simplification yields the given conditions and completes the proof.
\end{proof}
\begin{corollary}
Let $-1<B<A\leq1$. If either of the conditions \eqref{3a}--\eqref{3b} mentioned in \Cref{Theorem-Inclusion-Janowski-Class-NeP} hold true, then the Janowski convex class $\mathcal{C}[A,B]$ is contained in $\mathcal{C}_{Ne}$.
\end{corollary}
\section{Coefficient Estimates}  
\begin{theorem}\label{Theorem-Coeff-Est-Gutzmer-NeP}
Let $f(z)=z+\sum_{n=2}^{\infty}a_nz^n$ be a member of the function class $\mathcal{S}^*_{Ne}$. Then
\begin{align*}
\sum_{n=2}^{\infty}\left((3n)^2-7^2\right)|a_n|^2\leq{40}.
\end{align*}
\end{theorem}
\begin{proof}
Let $f(z)\in\mathcal{S}^*_{Ne}$. Then
\begin{align*}
\frac{zf'(z)}{f(z)}\prec\varphi_{\scriptscriptstyle{Ne}}(z)=1+z-\frac{z^3}{3}.
\end{align*}
Or
\begin{align*}
\frac{zf'(z)}{f(z)}=1+w(z)-\frac{w^3(z)}{3},
\end{align*}
where $w\in\mathcal{H}$ with $w(0)=0$ and $|w(z)|<1$ for all $z\in\mathbb{D}$. On taking $z=re^{it},\,r\in(0,1),\,t\in(-\pi,\pi]$, the above equation gives
\begin{align}\label{Eq1-Coeff-Est-Gutzmer-NeP}
f(re^{it})=\frac{3re^{it}f'(re^{it})}{3+3w(re^{it})-w^3(re^{it})}.
\end{align}
Using (\ref{Eq1-Coeff-Est-Gutzmer-NeP}) and Gutzmer's Theorem (see,\cite[p. 31, Problem 12]{Goodman-Book-UFs-I-1983}) 
we have
\begin{align*}
2\pi\sum_{n=1}^{\infty}|a_n|^2r^{2n}&=\int_{0}^{2\pi}\left|f(re^{it})\right|^2dt\\
               &=\int_{0}^{2\pi}\left|\frac{3re^{it}f'(re^{it})}{3+3w(re^{it})-w^3(re^{it})}\right|^2dt\\
                  &\geq9\int_{0}^{2\pi}\frac{\left|rf'(re^{it})\right|^2}{\left(3+3+1)\right)^2}dt\\
                  &=\frac{9r^2}{49}\int_{0}^{2\pi}\left|f'(re^{it})\right|^2dt\\
                   &=\frac{9r^2}{49}\int_{0}^{2\pi}f'(re^{it})\overline{f'(re^{it})}dt\\
     &=\frac{9r^2}{49}\int_{0}^{2\pi}\left(\sum_{n=1}^{\infty}
        na_nr^{n-1}e^{(n-1)it}\right)\left(\sum_{n=1}^{\infty}n\overline{a}_nr^{n-1}e^{-(n-1)it}\right)\\
   &=\frac{9}{49}\left(2\pi\sum_{n=1}^{\infty}n^2|a_n|^2r^{2n}\right),
\end{align*}
where $a_1=1$. Further simplification gives
\begin{align*}
\sum_{n=1}^{\infty}\left(\frac{9}{49}n^2-1\right)&|a_n|^2r^{2n}\leq{0}
\end{align*}
Letting $r\to{1^-}$ yields the desired result.
\end{proof}
\begin{corollary}
Let $f(z)=z+\sum_{n=2}^{\infty}a_nz^n\in\mathcal{S}^*_{Ne}$. Then
\begin{align*}
|a_n|\leq\sqrt{\frac{40}{(3n)^2-7^2}}\quad \text{ for } n=3,4,....
\end{align*}
\end{corollary}
Using the fact that $f\in\mathcal{C}_{Ne}\iff{zf'(z)}\in\mathcal{S}^*_{Ne}$, the following result is an immediate consequence of \Cref{Theorem-Coeff-Est-Gutzmer-NeP}.
\begin{theorem}
Let $f(z)=z+\sum_{n=2}^{\infty}a_nz^n\in\mathcal{C}_{Ne}$. Then
\begin{align*}
\sum_{n=2}^{\infty}n^2\left((3n)^2-7^2\right)|a_n|^2\leq{40}.
\end{align*}
\end{theorem}
\begin{corollary}
Let $f(z)=z+\sum_{n=2}^{\infty}a_nz^n\in\mathcal{C}_{Ne}$. Then
\begin{align*}
|a_n|\leq\frac{1}{n}\sqrt{\frac{40}{(3n)^2-7^2}}\quad \text{ for } n=3,4,....
\end{align*}
\end{corollary}
\begin{remark}
It is clear that $\varphi_{\scriptscriptstyle{Ne}}(z)=1+z-z^3/3\in\mathcal{H}^2$, the Hardy class of analytic functions in $\mathbb{D}$, defined as
\begin{align*}
\mathcal{H}^2=\left\{\sum_{n=0}^{\infty}{b_nz^n}:\sum_{n=0}^{\infty}|b_n|^2<\infty\right\}.
\end{align*}
Therefore, it may be concluded from Ma and Minda \cite[Theorem 4]{Ma-Minda-1992-A-unified-treatment} that for a function $f(z)=z+\sum_{n=2}^{\infty}a_nz^n\in\mathcal{S}^*_{Ne}\;(\text{or }\mathcal{C}_{Ne})$ the sharp order of growth is
$|a_n|=O(1/n)\;(\text{or }O(1/n^2))$.
\end{remark}
\begin{theorem}
\begin{enumerate}[{\rm(i)}]
\item Let $F_{n}(z)=z+a_nz^n\;(n=2,3,...)$. Then $F_{n}\in\mathcal{S}^*_{Ne}$ if, and only if,
\begin{align*}
|a_n|\leq\frac{2}{3n-1}, \quad n=2,3,....
\end{align*}
\item The function $K_A(z)=z/(1-Az)^2$ belongs to $\mathcal{S}^*_{Ne}$ if, and only if, $|A|\leq{1/4}$.
\item Let $|b|<1$. Then $L_b(z)=ze^{bz}\in\mathcal{S}^*_{Ne}$ if, and only if, $|b|\leq{2/3}$.
\end{enumerate}
\end{theorem}
\begin{proof}
(i) As $\mathcal{S}^*_{Ne}\subset\mathcal{S}^*$, we necessarily have $|a_n|\leq{1/n}$. 
It is easy to verify that the transformation
\begin{align*}
w(z)=\frac{zF'_{n}(z)}{F_{n}(z)}=\frac{1+na_nz^{n-1}}{1+a_nz^{n-1}}
\end{align*}
maps $\mathbb{D}$ onto the disk
\begin{align}\label{Eq-Disk(i)-Coeff-Est-NeP}
\left|w-\frac{1-n|a_n|^2}{1-|a_n|^2}\right|<\frac{(n-1)|a_n|}{1-|a_n|^2}.
\end{align}
Thus $F_{n}(z)$ belongs to the class $\mathcal{S}^*_{Ne}$ if, and only if, the disk (\ref{Eq-Disk(i)-Coeff-Est-NeP}) is contained in $\Omega_{Ne}$. Since
\begin{align*}
1-n|a_n|^2\leq1-|a_n|^2,
\end{align*}
it follows from \Cref{Lemma-Radius-ra-NeP} that the disk (\ref{Eq-Disk(i)-Coeff-Est-NeP}) is contained in $\Omega_{Ne}$ if, and only if,
\begin{align*}
\frac{(n-1)|a_n|}{1-|a_n|^2}\leq\frac{1-n|a_n|^2}{1-|a_n|^2}-\frac{1}{3}.
\end{align*}
That is, if, and only if, $|a_n|\leq2/(3n-1)$.\\
(ii) The fact that $\mathcal{S}^*_{Ne}\subset\mathcal{S}^*$ and $z/(1-z)^2\in\mathcal{S}^*$ lead us to conclude that $|A|\leq{1}$. This further gives
\begin{align}\label{Eq-GeqCond-Coeff-Est-NeP}
1+|A|^2\geq 1-|A|^2.
\end{align}
Also, the bilinear transformation
\begin{align*}
w(z)=\frac{zK'_{A}(z)}{K_{A}(z)}=\frac{1+Az}{1-Az}, \quad |A|\leq1
\end{align*}
maps $\mathbb{D}$ onto the disk
\begin{align}\label{Eq-Disk(ii)-Coeff-Est-NeP}
\left|w-\frac{1+|A|^2}{1-|A|^2}\right|<\frac{2|A|}{1-|A|^2}.
\end{align}
Therefore, $K_{A}(z)\in\mathcal{S}^*_{Ne}$ if, and only if, the disk (\ref{Eq-Disk(ii)-Coeff-Est-NeP}) lies inside the region $\Omega_{Ne}$. By (\ref{Eq-GeqCond-Coeff-Est-NeP}) and \Cref{Lemma-Radius-ra-NeP}, this is true if, and only if,
\begin{align*}
\frac{2|A|}{1-|A|^2}\leq\frac{5}{3}-\frac{1+|A|^2}{1-|A|^2},
\end{align*}
or equivalently, if, and only if, $|A|\leq1/4$.\\
(iii) Clearly, the linear transformation
$w(z)=zL'_{b}(z)/L_{b}(z)=1+bz$ maps $\mathbb{D}$ onto the disk
$\left|w-1\right|<|b|$. In view of \Cref{Lemma-Radius-ra-NeP}, this disk lies in the interior of $\Omega_{Ne}$ if, and only if, $|b|\leq{2/3}$.
\end{proof}
\begin{corollary}
\begin{enumerate}[{\rm(i)}]
\item $z+a_nz^n\in\mathcal{C}_{Ne}$ if, and only if,
\begin{align*}
|a_n|\leq\frac{2}{n(3n-1)}, \quad n=2,3,....
\end{align*}
\item $z/(1-Az)\in\mathcal{C}_{Ne}$ if, and only if, $|A|\leq{1/4}$.
\item $(e^{bz}-1)/b\in\mathcal{C}_{Ne}$ if, and only if, $|b|\leq{2/3}$.
\end{enumerate}
\end{corollary}
For $z\in\mathbb{D}$, define the class of functions $\mathcal{P}$ as
\begin{align*}
\mathcal{P}:=\left\{p(z)=1+\sum_{n=1}^{\infty}p_{n}z^n\in\mathcal{H}: \mathrm{Re}(p(z))>0 \right\}.
\end{align*}
\begin{lemma}
Let $p(z)=1+\sum_{n=1}^{\infty}p_{n}z^n\in\mathcal{P}$. Then \cite{Ma-Minda-1992-A-unified-treatment}
\begin{align}\label{Eq1-Lemma-of-p(z)-Coeff-Est-NeP}
|p_n|\leq{2} \text{ for } n\geq{1},
\end{align}
and for any complex number $\xi$, we have
\begin{align}\label{Eq2-Lemma-of-p(z)-Coeff-Est-NeP}
|{p_2}-\xi{p_1^2}|\leq{2\max\left\{1,|2\xi-1|\right\}}.
\end{align}
The results are sharp for the functions
\begin{align*}
p(z)=\frac{1+z}{1-z},\quad p(z)=\frac{1+z^2}{1-z^2}.
\end{align*}
\end{lemma}
\begin{theorem}
\begin{enumerate}[{\rm(i)}]
\item If $f(z)=z+{a_2}z^2+{a_3}z^3+\cdots\in\mathcal{S}^*_{Ne}$, then $|a_2|\leq{1},\,|a_3|\leq1/2$, and for any complex number
$\mu$,
\begin{align*}
|a_3-\mu{a_2^2}|\leq\frac{1}{2}\max\left\{1,|2\mu-1|\right\}.
\end{align*}
In particular, $|{a_3}-a_2^2|\leq1/2$. Equality holds for the function $f_{\scriptscriptstyle{Ne}}(z)$.
\item If $f(z)=z+{a_2}z^2+{a_3}z^3+\cdots\in\mathcal{C}_{Ne}$, then
$|a_2|\leq1/2$ and $|a_3|\leq1/6$. Equality holds for the function $\widehat{f}_{\scriptscriptstyle{Ne}}(z)$. Further, for any complex number $\mu$,
\begin{align*}
|a_3-\mu{a_2^2}|\leq\frac{1}{6}\max\left\{1,\left|\frac{3}{2}\mu-1\right|\right\}.
\end{align*}
\end{enumerate}
\end{theorem}
\begin{proof}
\begin{enumerate}[{\rm(i)}]
\item
Since $f\in\mathcal{S}^*_{Ne}$, we can find a Schwarz function $w(z)={w_1}z+{w_2}z^2+{w_3}z^3+\cdots$ such that
\begin{align}\label{Eq-Def-Implication-S*Ne-Coeff-Est-NeP}
\frac{zf'(z)}{f(z)}=\varphi_{\scriptscriptstyle{Ne}}\big(w(z)\big)=1+w(z)-\frac{1}{3}(w(z))^3.
\end{align}
Taking $w(z)=(p(z)-1)/(p(z)+1)$, where $p(z)=1+{p_1}z+{p_2}z^2+{p_3}z^3+\cdots\in\mathcal{P}$, and doing some simple calculations we obtain
\begin{align*}
w(z)=\frac{p_1 z}{2}+&\left(\frac{p_2}{2}-\frac{p_1^2}{4}\right) z^2+\frac{1}{8} \left(p_1^3-4 p_2 p_1+4 p_3\right) z^3 \\
            &+\frac{1}{16} \left(-p_1^4+6 p_2 p_1^2-8 p_3 p_1-4 p_2^2+8 p_4\right) z^4+\cdots.
\end{align*}
This gives
\begin{align}\label{Eq-Expansion-w(z)-Coeff-Est-NeP}
1+w(z)-\frac{1}{3}(w(z))^3=1+\frac{p_1}{2}z+&\left(\frac{p_2}{2}-\frac{p_1^2}{4}\right) z^2
                                 +\left(\frac{p_1^3}{12}-\frac{p_1 p_2}{2}+\frac{p_3}{2}\right) z^3 \nonumber\\
                        &+\left(\frac{1}{4} p_1^2p_2-\frac{p_1 p_3}{2}-\frac{p_2^2}{4}+\frac{p_4}{2}\right) z^4+\cdots.
\end{align}
Also it is easy to check that for $f(z)=z+{a_2}z^2+{a_3}z^3+\cdots$,
\begin{align}\label{Eq-Expansion-zf'/f-Coeff-Est-NeP}
\frac{zf'(z)}{f(z)}=1+a_2 z+&\left(2 a_3-a_2^2\right) z^2+\left(a_2^3-3 a_2 a_3+3 a_4\right) z^3 \nonumber\\
                        &+\left(-a_2^4+4 a_2^2a_3 -4 a_2 a_4-2 a_3^2+4 a_5\right) z^4+\cdots.
\end{align}
Using (\ref{Eq-Expansion-w(z)-Coeff-Est-NeP}) and (\ref{Eq-Expansion-zf'/f-Coeff-Est-NeP}) in the equation (\ref{Eq-Def-Implication-S*Ne-Coeff-Est-NeP}), and comparing coefficients we get
\begin{align}\label{Eqs-Bounds-a2a3-Coeff-Est-NeP}
a_2=\frac{p_1}{2},\quad a_3=\frac{p_2}{4}.
\end{align}
Applying (\ref{Eq1-Lemma-of-p(z)-Coeff-Est-NeP}) in (\ref{Eqs-Bounds-a2a3-Coeff-Est-NeP}), we obtain $|a_2|\leq{1}$ and $|a_3|\leq1/2$.
Also, from (\ref{Eqs-Bounds-a2a3-Coeff-Est-NeP}) and (\ref{Eq2-Lemma-of-p(z)-Coeff-Est-NeP}), we have
\begin{align*}
|{a_3}-\mu{a_2^2}|=\frac{1}{4}|{p_2}-\mu{p_1^2}|\leq\frac{1}{2}\max\left\{1,|2\mu-1|\right\}.
\end{align*}
\item
Similarly, $f\in\mathcal{C}_{Ne}$ implies
\begin{align}\label{Eq-Def-Implication-CNe-Coeff-Est-NeP}
1+\frac{zf''(z)}{f'(z)}=1+w(z)-\frac{1}{3}(w(z))^3.
\end{align}
Using (\ref{Eq-Expansion-w(z)-Coeff-Est-NeP}) and the expansion
\begin{align*}
1+\frac{zf''(z)}{f'(z)}=1+2 a_2 z+&\left(6 a_3-4 a_2^2\right) z^2+\left(8 a_2^3-18 a_3 a_2+12 a_4\right) z^3\\
                            &+\left(-16 a_2^4+48 a_3 a_2^2-32 a_4 a_2-18 a_3^2+20 a_5\right) z^4+\cdots
\end{align*}
in (\ref{Eq-Def-Implication-CNe-Coeff-Est-NeP}), and comparing we get $a_2=p_1/4$ and $a_3=p_2/12$. Now, as in part (i), the results easily follow. \qedhere
\end{enumerate}
\end{proof}
\begin{remark}
From \Cref{Theorem-Inclusion-Relations-NeP} (a), we have $\mathcal{S}^*_{Ne}\subset\mathcal{S}^*(1-2\sqrt{2}/3)$. A result in Schild \cite{Schild-1965-Starlike-order-Alpha} yields that for every function $f(z)=z+\sum_{n=2}^{\infty}a_nz^n$ in $\mathcal{S}^*_{Ne}$, we have
\begin{align*}
|a_n|\leq\frac{1}{(n-1)!}\prod_{k=2}^{n}\left(k-2+\frac{4\sqrt{2}}{3}\right),\quad n\geq2.
\end{align*}
Further, Alexander type theorem shows that if $f(z)=z+\sum_{n=2}^{\infty}a_nz^n\in\mathcal{C}_{Ne}$, then
\begin{align*}
|a_n|\leq\frac{1}{n!}\prod_{k=2}^{n}\left(k-2+\frac{4\sqrt{2}}{3}\right),\quad n\geq2.
\end{align*}
\end{remark}
\section{Subordination Implications} 
\begin{definition}[\cite{Miller-Mocanu-Book-2000-Diff-Sub}]
Let $\Psi:\mathbb{C}^2\times\mathbb{D}\to\mathbb{C}$ be a complex function, and let $h:\mathbb{D}\to\mathbb{C}$ be univalent. If $p\in\mathcal{H}$ satisfies the first-order differential subordination
\begin{align}\label{Def-Diff-Subord-Psi}
\Psi(p(z),zp'(z);z)\prec{h(z)},\quad z\in\mathbb{D},
\end{align}
then $p$ is called a solution of the differential subordination {\rm(\ref{Def-Diff-Subord-Psi})}. If $q:\mathbb{D}\to\mathbb{C}$ is univalent and $p{\prec}q$ for all $p$ satisfying {\rm(\ref{Def-Diff-Subord-Psi})}, then $q(z)$ is said to be a dominant of {\rm(\ref{Def-Diff-Subord-Psi})}. A dominant $\tilde{q}$ that satisfies $\tilde{q}\prec q$ for all dominants $q$ of {\rm(\ref{Def-Diff-Subord-Psi})} is called the best dominant of {\rm(\ref{Def-Diff-Subord-Psi})}. The best dominant is unique up to the rotations of $\mathbb{D}$.
\end{definition}
Let $p(z)$ be an analytic function satisfying $p(0)=1$. Nunokawa et al. \cite{Nunokawa-Owa-1989-One-criterion-PAMS} proved that the differential subordination $1+zp'(z)\prec1+z$ implies $p(z)\prec1+z$ and, as a result, gave a criterion for $f\in\mathcal{A}$ to be univalent. Ali et al. \cite{Ali-Ravi-2007-Janowski-Starlikeness-IJMMS} replaced $1+z$ by $(1+Dz)/(1+Ez)$ and determined the conditions on $\beta\in\mathbb{R}$ in terms of $A,B,D,E\in[-1,1]\,(B<A \text{ and } E<D)$ so that the subordination $1+\beta{zp'(z)/p^j(z)}\prec(1+Dz)/(1+Ez)\,(j=0,1,2)$ implies $p(z)\prec(1+Az)/(1+Bz)$ and, further, established some sufficient conditions for Janowski starlikeness. Ali et al. \cite{Ali-Ravi-2012-Diff-Sub-LoB-TaiwanJM} determined the conditions on $\beta$ so that the subordination $p(z)\prec\sqrt{1+z}$ holds true whenever
$1+\beta{zp'(z)/p^j(z)}\prec\sqrt{1+z},\,j=0,1,2$. Kumar and Ravichandran \cite{SushilKumar-Ravi-2018-Sub-Positive-RP-CAOT} gave sharp bounds for $\beta$ in order to guarantee the following subordination implication:
\begin{align*}
1+\beta\frac{zp'(z)}{p^j(z)}&\prec\varphi_{\scriptscriptstyle{0}}(z),\,\sqrt{1+z},\,1+\sin{z}{\implies} p(z)\prec{e^z}.
\end{align*}
As a consequence, they also established some sufficient conditions for the function class $\mathcal{S}^*_e$. Recently, Ahuja et al. \cite{Ahuja-Ravi-2018-App-Diff-Sub-Stud-Babe-Bolyai} obtained sharp bounds on $\beta$ so that the differential subordination $1+\beta{zp'(z)/p^j(z)}\prec\sqrt{1+z}\,(j=0,1,2)$ implies $p(z)\prec:\sqrt{1+z},\,
\varphi_{\scriptscriptstyle {0}}(z),\,1+\sin{z},\,\sqrt{1+z^2}+z,\,1+\frac{4}{3}z+\frac{2}{3}z^2,\,
\frac{1+Az}{1+Bz}$, where $-1<A<B<1$.
Similar problems have been considered in \cite{Cho-Ravi-2018-Diff-Sub-Booth-Lem-TJM,Goel-Siva-2019-Sigmoid-BMMS,
Siva-Govind-Muru-2019-Leaf-Like-JCAA}.\\
In this section, we will discuss certain implications of the first order differential subordination
$1+\beta{zp'(z)}/{p^j(z)}\prec\varphi_{\scriptscriptstyle{Ne}}(z)=1+z-{z^3}/{3},\, j=0,1,2$. As a consequence, we obtain sufficient conditions for $f\in\mathcal{A}$ to belong to some of the early mentioned Ma-Minda type function classes. To prove our results, we will make use of the following lemma.
\begin{lemma}[{\cite[Theorem 3.4h, p. 132]{Miller-Mocanu-Book-2000-Diff-Sub}}]\label{Lemma-3.4h-p132-Miller-Mocanu-Book}
Let $q$ be univalent in $\mathbb{D}$, and let $\psi$ and $\nu$ be analytic in a domain $D$ containing $q(\mathbb{D})$ with $\psi(w)\neq0$ when $w\in{q(\mathbb{D})}$. Set
\begin{align*}
Q(z):=zq'(z)\psi\big(q(z)\big) \quad \text{ and } \quad  h(z):=\nu\big(q(z)\big)+Q(z).
\end{align*}
Suppose that either
\begin{enumerate}[{\rm(i)}]
\item $h$ is convex, or
\item $Q$ is starlike.\\
In addition, assume that
\item $\mathrm{Re}\left(\frac{zh'(z)}{Q(z)}\right)>0$.
\end{enumerate}
If $p\in\mathcal{H}$ with $p(0)=q(0)$, $p(\mathbb{D})\subset{D}$ and
\begin{align*}
\nu\big(p(z)\big)+zp'(z)\psi\big(p(z)\big)\prec\nu\big(q(z)\big)+zq'(z)\psi\big(q(z)\big),
\end{align*}
then $p\prec{q}$, and $q$ is the best dominant.
\end{lemma}
\subsection*{Implications of $1+\beta{zp'(z)}\prec\varphi_{\scriptscriptstyle{Ne}}(z)$}
\begin{theorem}\label{Thrm-Subord-Implications-Neph-j0-NeP}
Let $p\in\mathcal{H}$ such that $p(0)=1$, and let $1+\beta{zp'(z)}\prec\varphi_{\scriptscriptstyle{Ne}}(z)$. Then the following subordinations hold:
\begin{enumerate}[{\rm(a)}]
\item\label{Subord-a} $p(z)\prec\varphi_{\scriptscriptstyle{\leftmoon}}(z)=\sqrt{1+z^2}+z$ whenever
           $\beta\geq\frac{4}{9} \left(2+\sqrt{2}\right)\approx1.51743$.
\item\label{Subord-b} $p(z)\prec\varphi_{\scriptscriptstyle{C}}(z)=1+\frac{4}{3}z+\frac{2}{3}z^2$ whenever $\beta\geq\frac{4}{3}$.
\item\label{Subord-c} $p(z)\prec\varphi_{\scriptscriptstyle{lim}}(z)=1+\sqrt{2}z+\frac{1}{2}z^2$ whenever $\beta\geq\frac{16}{63}(1+2\sqrt{2})\approx0.972299$.
\item\label{Subord-d} $p(z)\prec{e}^z$ whenever $\beta\geq\frac{8 e}{9 (e-1)}\approx1.4062$.
\item\label{Subord-e} $p(z)\prec\varphi_{\scriptscriptstyle{S}}(z)=1+\sin{z}$ whenever $\beta\geq\frac{8}{9\sin1}\approx1.05635$.
\item\label{Subord-f} $p(z)\prec\varphi_{\scriptscriptstyle{0}}(z)=1+\frac{z}{k}\left(\frac{k+z}{k-z}\right)\; (k=1+\sqrt{2})$ whenever $\beta\geq\frac{8}{9}\left(3+2\sqrt{2}\right)\approx5.18082$.
\item\label{Subord-g} $p(z)\prec\varphi_{\scriptscriptstyle{Ne}}(z)$ whenever $\beta\geq\frac{4}{3}$.
\end{enumerate}
The bounds on $\beta$ are sharp.
\end{theorem}
\begin{proof}
Consider the analytic function $q_{\beta}:\overline{\mathbb{D}}\to\mathbb{C}$ defined by
\begin{align*}
q_{\beta}(z)=1+\frac{1}{\beta}\left(z-\frac{z^3}{9}\right).
\end{align*}
Clearly, the function $q_{\beta}(z)$ is a solution of the first order differential equation $1+\beta{zq'_{\beta}(z)}=\varphi_{\scriptscriptstyle{Ne}}(z)$. Defining $\nu(w)=1$ and $\psi(w)=\beta$, the functions $Q(z)$ and $h(z)$ in \Cref{Lemma-3.4h-p132-Miller-Mocanu-Book} become
\begin{align*}
Q(z)=zq'_{\beta}(z)\psi\big(q_{\beta}(z)\big)
                ={\beta}zq'_{\beta}(z)=z-\frac{z^3}{3}=\varphi_{\scriptscriptstyle{Ne}}(z)-1
\end{align*}
and
\begin{align*}
h(z)=\nu\big(q_\beta(z)\big)+Q(z)=1+Q(z)=\varphi_{\scriptscriptstyle {Ne}}(z).
\end{align*}
Since the function $\varphi_{\scriptscriptstyle{Ne}}(z)-1=z-z^3/3$ is starlike in $\mathbb{D}$
\cite[Theorem 2.3]{Brannan-1970-On-Univalent-Polynomials}, the function $Q(z)$ is starlike in $\mathbb{D}$. Further, $\mathrm{Re}\left(zh'(z)/Q(z)\right)=\mathrm{Re}\left(zQ'(z)/Q(z)\right)>0$ for $z\in\mathbb{D}$ and $p(0)=1=q_{\beta}(0)$. Therefore, it follows from  \Cref{Lemma-3.4h-p132-Miller-Mocanu-Book} that the subordination $1+{\beta}zp'(z)\prec1+{\beta}zq'_{\beta}(z)=\varphi_{\scriptscriptstyle{Ne}}(z)$ implies $p(z)\prec{q_{\beta}(z)}$. Now, each of the desired conclusions (\ref{Subord-a})--(\ref{Subord-g}) will follow, sequentially, if $q_{\beta}(z)\prec\varphi_{\scriptscriptstyle{\leftmoon}}(z),
\varphi_{\scriptscriptstyle{C}}(z),\varphi_{\scriptscriptstyle{lim}}(z),e^z,
\varphi_{\scriptscriptstyle{S}}(z),\varphi_{\scriptscriptstyle{0}}(z),\varphi_{\scriptscriptstyle{Ne}}(z)$.
\begin{enumerate}[(a):]
\item\label{Proof-Thrm-Subord-Neph-Impl-Crscnt-j0-NeP}
The necessary condition for the subordination $q_\beta(z)\prec\varphi_{\scriptscriptstyle{\leftmoon}}(z)$ to hold is that
\begin{align}\label{N&S-Cond-Subord-Neph-Impl-Crscnt-j0-NeP}
\sqrt{2}-1=\varphi_{\scriptscriptstyle{\leftmoon}}(-1)<q_\beta(-1)
                        <q_\beta(1)<\varphi_{\scriptscriptstyle{\leftmoon}}(1)=\sqrt{2}+1.
\end{align}
Graphical observation reveals that the condition (\ref{N&S-Cond-Subord-Neph-Impl-Crscnt-j0-NeP}) is also sufficient for the subordination $q_\beta(z)\prec\varphi_{\scriptscriptstyle{\leftmoon}}(z)$ to hold. Simplifying (\ref{N&S-Cond-Subord-Neph-Impl-Crscnt-j0-NeP}), we obtain $\beta\geq{4\left(2+\sqrt{2}\right)/9}=\beta_1$ and $\beta\geq{4\left(2-\sqrt{2}\right)/9}=\beta_2$. Therefore, the subordination $q_\beta(z)\prec\varphi_{\scriptscriptstyle{\leftmoon}}(z)$ holds true provided $\beta\geq\max\{\beta_1,\beta_2\}=\beta_1$. The number $\beta_1=4\left(2+\sqrt{2}\right)/9$ can not be decreased (see \Cref{Figure-Subord-Neph-Impl-Crscnt-NeP}).
\end{enumerate}
Following the same procedure, the other parts are easy to prove.
\end{proof}
Taking $p(z)=zf'(z)/f(z),\;f\in\mathcal{A}$ in the \Cref{Thrm-Subord-Implications-Neph-j0-NeP}, we obtain the following result.
\begin{corollary}
If $f\in\mathcal{A}$ satisfies
\begin{align*}
1+\beta\frac{zf'(z)}{f(z)}\left(1-\frac{zf'(z)}{f(z)}+\frac{zf''(z)}{f'(z)}\right)
                                \prec\varphi_{\scriptscriptstyle{Ne}}(z),
\end{align*}
then
\begin{enumerate}[{\rm(a)}]
\item $f\in\mathcal{S}^*_{\leftmoon}(z)$ whenever $\beta\geq\frac{8}{9 \left(2-\sqrt{2}\right)}$.
\item $f\in\mathcal{S}^*_C$ whenever $\beta\geq\frac{4}{3}$.
\item $f\in\mathcal{S}^*_{lim}$ whenever $\beta\geq\frac{16}{9(2\sqrt{2}-1)}$.
\item $f\in\mathcal{S}^*_{e}$ whenever $\beta\geq\frac{8 e}{9 (e-1)}$.
\item $f\in\mathcal{S}^*_{S}$ whenever $\beta\geq\frac{8}{9\sin1}$.
\item $f\in\mathcal{S}^*_{R}$ whenever $\beta\geq\frac{8}{9}\left(3+2\sqrt{2}\right)$.
\item $f\in\mathcal{S}^*_{Ne}$ whenever $\beta\geq\frac{4}{3}$.
\end{enumerate}
The bounds on $\beta$ are best possible.
\end{corollary}
\begin{minipage}{0.58\textwidth}
  \centering
   \includegraphics{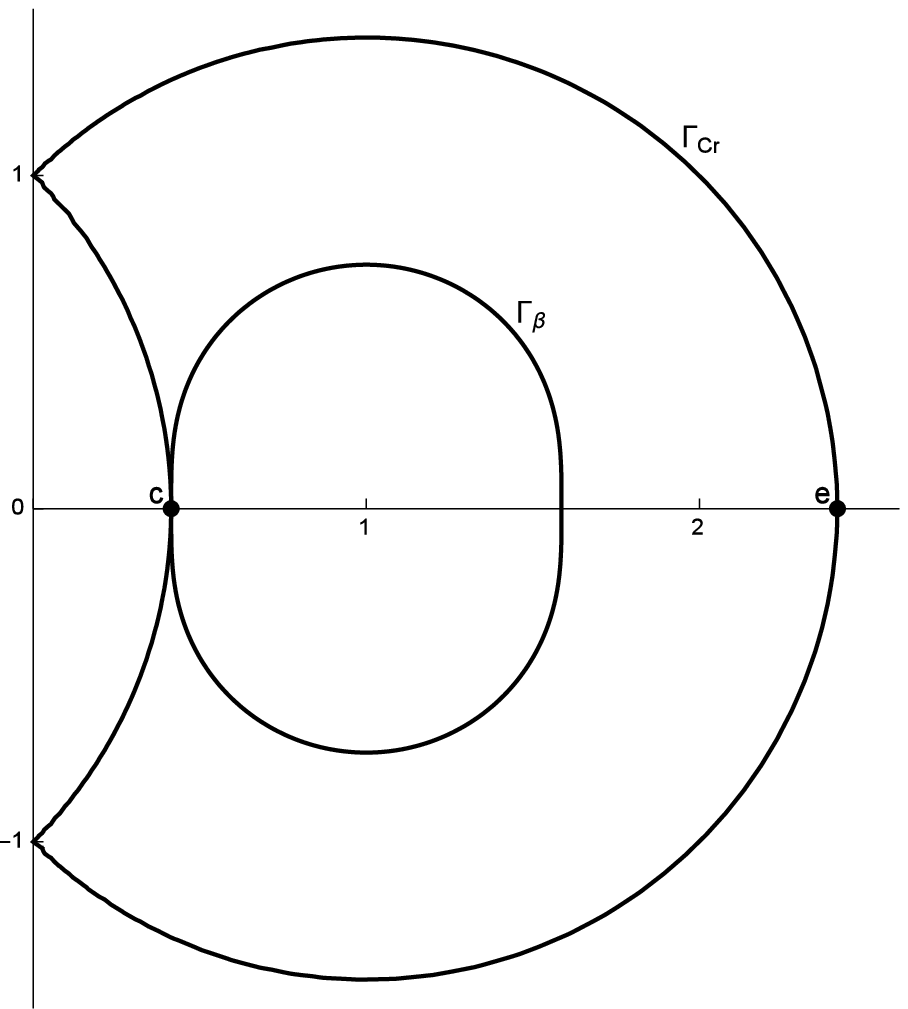}
\end{minipage}
\;
\begin{minipage}{0.42\textwidth}
   $\Gamma_{Cr}:\left|w^2-1\right|=2 \left|w\right|, \mathrm{Re}\,w>0$.\\
   \\
    $\Gamma_\beta:$ Boundary curve of $q_{\beta}(\mathbb{D})$,\\ where $q_{\beta}(z)=1+\frac{1}{\beta}\left(z-\frac{z^3}{9}\right)$\\
     with $\beta=4\left(2+\sqrt{2}\right)/9$.\\
    \\
     $c=\sqrt{2}-1$.\\
      $e=\sqrt{2}+1$.\\
\end{minipage}
\captionof{figure}{Sharpness of $\beta=\frac{4}{9} \left(2+\sqrt{2}\right)$ in \Cref{Thrm-Subord-Implications-Neph-j0-NeP}(\ref{Subord-a}).}
\label{Figure-Subord-Neph-Impl-Crscnt-NeP}
\noindent
\subsection*{Implications of $1+\beta\frac{zp'(z)}{p(z)}\prec\varphi_{\scriptscriptstyle{Ne}}(z)$}
\begin{theorem}\label{Thrm-Subord-Implications-Neph-j1-NeP}
Let $p\in\mathcal{H}$ satisfies $p(0)=1$, and let
\begin{align*}
1+\beta\frac{zp'(z)}{p(z)}\prec\varphi_{\scriptscriptstyle{Ne}}(z).
\end{align*}
Then the following subordinations hold:
\begin{enumerate}[{\rm(a)}]
\item $p(z)\prec\varphi_{\scriptscriptstyle{S}}(z)$ whenever $\beta\geq\frac{8/9}{\log(1+\sin1)}\approx1.45585$.
\item $p(z)\prec\varphi_{\scriptscriptstyle{0}}(z)$ whenever $\beta\geq-\frac{8/9}{\log\left(2\sqrt{2}-2\right)}\approx4.72245$.
\item $p(z)\prec\varphi_{\scriptscriptstyle{Ne}}(z)$ whenever $\beta\geq\frac{8/9}{\log \left(5/3\right)}$.
\end{enumerate}
The bounds on $\beta$ are best possible.
\end{theorem}
\begin{proof}
Define the function $q_{\beta}:\overline{\mathbb{D}}\to\mathbb{C}$ as
\begin{align*}
q_{\beta}(z)=\exp\left(\frac{1}{\beta}\left(z-\frac{z^3}{9}\right)\right).
\end{align*}
The function $q_{\beta}(z)$ is analytic and satisfies $1+\beta{zq'_{\beta}(z)}/{q_{\beta}(z)}=\varphi_{\scriptscriptstyle{Ne}}(z)$. Taking $\nu(w)=1$ and $\psi(w)=\beta/w$ in \Cref{Lemma-3.4h-p132-Miller-Mocanu-Book}, we see that \Cref{Lemma-3.4h-p132-Miller-Mocanu-Book} is applicable. Therefore the subordination $1+\beta{zp'(z)}/{p(z)}\prec1+\beta{zq'_{\beta}(z)}/{q_{\beta}(z)}=\varphi_{\scriptscriptstyle{Ne}}(z)$ implies
$p(z)\prec{q_{\beta}(z)}$. Now, to prove the conclusions (a), (b) and (c), we only need to prove that $q_{\beta}(z)\prec:\varphi_{\scriptscriptstyle{S}}(z),\varphi_{\scriptscriptstyle{0}}(z)$ and $\varphi_{\scriptscriptstyle{Ne}}(z)$, respectively.
\begin{enumerate}[(a):]
\item As in \Cref{Thrm-Subord-Implications-Neph-j0-NeP}(\ref{Proof-Thrm-Subord-Neph-Impl-Crscnt-j0-NeP}), the necessary and sufficient condition for the subordination $q_\beta(z)\prec\varphi_{\scriptscriptstyle{S}}(z)$ to hold is that
$1-\sin{1}<q_\beta(-1)<q_\beta(1)<1+\sin{1}$. This is true if
\begin{align*}
\beta\geq\frac{-8/9}{\log(1-\sin1)}=\beta_1 \text{ and } \beta\geq\frac{8/9}{\log(1+\sin1)}=\beta_2.
\end{align*}
Thus, the subordination $q_{\beta}(z)\prec\varphi_{\scriptscriptstyle{S}}(z)$ holds if $\beta\geq\max\left\{\beta_1,\beta_2\right\}=\beta_2$.
\end{enumerate}
The subordinations (b) and (c) follow similarly.
\end{proof}
\begin{corollary}
If $f\in\mathcal{A}$ satisfies
\begin{align*}
1+\beta\left(1-\frac{zf'(z)}{f(z)}+\frac{zf''(z)}{f'(z)}\right)\prec\varphi_{\scriptscriptstyle{Ne}}(z),
\end{align*}
then
\begin{enumerate}[{\rm(a)}]
\item $f\in\mathcal{S}^*_{S}$ whenever $\beta\geq\frac{8/9}{\log(1+\sin1)}$.
\item $f\in\mathcal{S}^*_{R}$ whenever $\beta\geq-\frac{8/9}{\log\left(2\sqrt{2}-2\right)}$.
\item $f\in\mathcal{S}^*_{Ne}$ whenever $\beta\geq\frac{8/9}{\log \left(5/3\right)}$.
\end{enumerate}
The bounds on $\beta$ are best possible.
\end{corollary}

\subsection*{Implications of $1+\beta\frac{zp'(z)}{p^2(z)}\prec\varphi_{\scriptscriptstyle{Ne}}(z)$}
\begin{theorem}\label{Thrm-Subord-Implications-Neph-j2-NeP}
Let $p\in\mathcal{H}$ such that $p(0)=1$, and let
\begin{align*}
1+\beta\frac{zp'(z)}{p^2(z)}\prec\varphi_{\scriptscriptstyle{Ne}}(z).
\end{align*}
Then the following subordinations hold:
\begin{enumerate}[{\rm(a)}]
\item $p(z)\prec\varphi_{\scriptscriptstyle{\leftmoon}}(z)$ whenever $\beta\geq\frac{4}{9}\left(2+\sqrt{2}\right)\approx1.51743$.
\item $p(z)\prec\varphi_{\scriptscriptstyle{C}}(z)$ whenever $\beta\geq\frac{4}{3}$.
\item $p(z)\prec\varphi_{\scriptscriptstyle{lim}}(z)$ whenever $\beta\geq\frac{8}{63}\left(5+4\sqrt{2}\right)\approx1.35325$.
\item $p(z)\prec{e}^z$ whenever $\beta\geq\frac{8e}{9(e-1)}\approx1.4062$.
\item $p(z)\prec\varphi_{\scriptscriptstyle{S}}(z)$ whenever $\beta\geq\frac{8(1+\sin1)}{9\sin1}\approx1.94524$.
\item $p(z)\prec\varphi_{\scriptscriptstyle{0}}(z)$ whenever $\beta\geq\frac{16}{9}\left(1+\sqrt{2}\right)\approx4.29194$.
\item $p(z)\prec\varphi_{\scriptscriptstyle{Ne}}(z)$ whenever $\beta\geq\frac{20}{9}$.
\end{enumerate}
The bounds on $\beta$ are best possible.
\end{theorem}
\begin{proof}
Define $q_{\beta}:\overline{\mathbb{D}}\to\mathbb{C}$ as
\begin{align*}
q_{\beta}(z)=\left(1-\frac{z}{\beta}\left(1-\frac{z^2}{9}\right)\right)^{-1}.
\end{align*}
The function $q_{\beta}(z)$ is analytic and satisfies the first order differential equation $1+\beta{zq'_{\beta}(z)}/{q^2_{\beta}(z)}=\varphi_{\scriptscriptstyle{Ne}}(z)$. Now take the functions $\nu$ and $\psi$ of \Cref{Lemma-3.4h-p132-Miller-Mocanu-Book} as $\nu(w)=1$ and $\psi(w)=\beta/w^2$, and proceed as in \Cref{Thrm-Subord-Implications-Neph-j0-NeP}.
\end{proof}

\begin{corollary}\label{Coroll-Suff-Cond-Subor-NeP}
If $f\in\mathcal{A}$ satisfies
\begin{align*}
1+\beta\left(\frac{zf'(z)}{f(z)}\right)^{-1}\left(1-\frac{zf'(z)}{f(z)}+\frac{zf''(z)}{f'(z)}\right)
                                                  \prec\varphi_{\scriptscriptstyle{Ne}}(z),
\end{align*}
then
\begin{enumerate}[{\rm(a)}]
\item $f\in\mathcal{S}^*_{\leftmoon}(z)$ whenever $\beta\geq\frac{4}{9}\left(2+\sqrt{2}\right)$.
\item $f\in\mathcal{S}^*_C$ whenever $\beta\geq\frac{4}{3}$.
\item $f\in\mathcal{S}^*_{lim}$ whenever $\beta\geq\frac{8}{63}\left(5+4\sqrt{2}\right)$.
\item $f\in\mathcal{S}^*_{e}$ whenever $\beta\geq\frac{8e}{9(e-1)}$.
\item $f\in\mathcal{S}^*_{S}$ whenever $\beta\geq\frac{8(1+\sin1)}{9\sin1}$.
\item $f\in\mathcal{S}^*_{R}$ whenever $\beta\geq\frac{16}{9}\left(1+\sqrt{2}\right)$.
\item $f\in\mathcal{S}^*_{Ne}$ whenever $\beta\geq\frac{20}{9}$.
\end{enumerate}
The bounds on $\beta$ are best possible.
\end{corollary}

\end{document}